\documentclass[11pt,english]{smfart}

\usepackage[latin1]{inputenc}
\usepackage[T1]{fontenc}
\usepackage[english,francais]{}
\usepackage{hyperref}
\usepackage{amssymb,amsmath,url,xspace,smfthm,euscript,enumerate}

\def\notdiv{\nmid}

\def\too{\relbar\lien\rightarrow}
\def\tooo{\relbar\lien\relbar\lien\too}

\let\ds=\displaystyle
\def\lien{\mathrel{\mkern-4mu}}
\let\st=\scriptstyle

  \def\Q{\mathbb{Q}}
  \def\Z{\mathbb{Z}}
  \def\F{\mathbb{F}}

 \def\zz{\mathbb{Z}}

\def\virg{\raise 2pt \hbox{,}\,\,}

\def\No{{\rm N}}
\def\Sauf{\!\setminus\!}

\def\Cl{{\mathcal C}\hskip-2pt{\ell}}

\def\cl{c\hskip-1pt{\ell}}
\def\Pl{P\hskip-1pt{l}}

\def\Frac#1#2{\hbox{\footnotesize $\displaystyle \frac{#1}{#2}$}}
\def\plus{\ds\mathop{\raise 2.0pt \hbox{$\bigoplus $}}\limits}
\def\mult{\ds\mathop{\raise 2.0pt \hbox{$\bigotimes$}}\limits}
\def\prd{ \ds\mathop{\raise 2.0pt \hbox{$  \prod   $}}\limits}
\def\Cap{ \ds\mathop{\raise 2.0pt \hbox{$\bigcap   $}}\limits}
\def\Cup{ \ds\mathop{\raise 2.0pt \hbox{$\bigcup   $}}\limits}
\def\sm{  \ds\mathop{\raise 2.0pt \hbox{$  \sum    $}}\limits}

\def\fin{\vbox{\hrule\hbox to 7.2pt{\vrule height 7pt\hfil\vrule}\hrule}}
\def \tensorZp{\otimes{\raise -0.8pt \hbox{\!\!$_{_{\zz_{\!p}}}$}}}
\def \tensorZ{\otimes{\raise -0.8pt \hbox{\!\!$_{_{\zz}}$}}}

\def\Cup{\displaystyle \mathop{\raise 2.0pt \hbox{$\bigcup$}}\limits}
 
\def\wt{\widetilde}	\def\ov{\overline}  
\let\st=\scriptstyle

\def\fin{\vbox{\hrule\hbox to 7.2pt{\vrule height 7pt\hfil\vrule}\hrule}}

\theoremstyle{plain}

\title{Analysis of the classical cyclotomic approach \\ to fermat${}'$s last theorem}
\author{Georges Gras}

\address{Villa la Gardette, chemin Ch\^ateau Gagni\`ere,
F-38520 Le Bourg d'Oisans}
\email{g.mn.gras@wanadoo.fr}
\urladdr{http://monsite.orange.fr/maths.g.mn.gras/}

\newcounter{note}
 \newcommand{\zero}{\setcounter{note}{0}}
 \newcommand{\ntm}{\footnotemark\addtocounter{note}{1}}
 \newcommand{\initnt}{\addtocounter{footnote}{-\value{note}}}
 \newcommand{\ntt}[1]{\addtocounter{footnote}{1}
    \footnotetext{#1}}

\begin{document}
\newenvironment{nota}{\begin{enonce}{Notation}}{\end{enonce}}

\newenvironment{exa}{\begin{enonce}{Example}}{\end{enonce}}

\frontmatter

\date{March 14, 2010}

\thanks{The author thanks Christian Maire
for his interest and comments concerning this didactic paper, Roland Qu\^eme 
for an observation on Wieferich${}'$s criterion,
and the Referee for his valuable help and for the corrections 
of  english.}
\keywords{Fermat${\,}'\!$s last theorem, Class field theory, Cyclotomic fields,
Reflection theo\-rems, Radicals, Gauss sums}

\subjclass{11D41, 11R18, 11R37, 11R29}

\begin{abstract}
We give again the proof of several classical results concerning the 
cycloto\-mic approach to Fermat${\,}'$s last theorem
using {\it exclusively}\, class field theory (essentially the reflection theorems),
without any calculations.
The fact that this is possible suggests a part of the logical inefficiency of the 
historical investigations.

\noindent
We analyze the significance of the numerous computations
of the literature, to show how they are probably too local to get any proof of
the theorem. However
we use the derivation method of Eichler as a prerequisite for our 
purpose, a method which is also local but more effective.

\noindent
Then we propose some modest ways of study in a more diophantine context
using radicals; this point of view would require
further nonalgebraic investigations.
\end{abstract}

\begin{altabstract}
Nous redonnons la preuve de plusieurs r\'esultats classiques concernant
l'approche cyclotomique du th\'eor\`eme de Fermat
en utilisant {\it exclusivement} la th\'eorie du corps de classes
(notamment les th\'eor\`emes de r\'eflexion), sans aucun calcul.
Le fait que ceci soit possible sugg\`ere une part d'inefficacit\'e logique des
investigations historiques.

\noindent
Nous analysons la signification de nombreux calculs de
la litt\'erature, afin de montrer en quoi ils sont
probablement trop locaux pour donner une preuve du th\'eor\`eme.
Cependant
nous  utilisons la m\'ethode de d\'erivation d'Eichler
 comme pr\'ealable \`a notre
d\'emarche, m\'ethode aussi locale, mais plus effective.

\noindent
Ensuite, nous proposons quelques modestes voies d'\'etude,
dans un contexte plus diophantien,
utilisant des radicaux, point de vue qui n\'ecessiterait
d'\'etablir de nouvelles propri\'et\'es non alg\'ebriques.
\end{altabstract}

\maketitle

\mainmatter

\section*{Introduction and Generalities}

The classical approaches to Fermat${}'$s last 
theorem (FLT) are essentially of a $p$-adic nature in the $p$th cyclotomic 
field; thus these studies turn to be arithmetic modulo $p$, in 
which case the distinction between first and second case is
necessary but unnatural as Wiles${}'$s proof suggests.
Even if the starting point is of a global nature ($p$th powers of ideals, 
classes, units, logarithmic derivative of Eichler,\ldots), the conclusion of the study is 
mostly local (congruences modulo~$p$) as we can see for instance in 
Ribenboim and Washington${}'$s books [R, Wa].

\smallskip
We don${}'$t know (for instance in the first case of FLT) if $p$-adic 
investigations (Kummer${}'$s congruences, Mirimanoff or Thaine${}'$s congruences,
Wieferich or Wendt${}'$s criteria,\ldots) are able, from a logical point of view, to 
succeed in proving it. We think that probably not and we think that 
all these dramatically numerous necessary conditions can, in some 
sense, be satisfied in a very rare ``\,numerical setting\,'',  as for the 
question of
Vandiver${}'$s conjecture for which we have given a probabilistic study in [Gr1, 
 II.5.4.9.2]: the number of favourable primes less than $p$
 (for a counterexample) can be of the
form $c\,.\, {\rm log}({\rm log}(p))$, $c < 1$.

\smallskip
This is to be relativized with the result
of Soul\'e [S] showing (after that of Kurihara [Ku] for $n=3$)
that for odd $n$, the real components $\Cl_{\omega^{p-n}}$
of the $p$-class group\,\footnote{Standard definitions
with the character of Teichm\"uller $\omega$ and the
corresponding eigenspaces $\Cl_{\omega^i}$,
also denoted $\Cl^{(i)}$, $i=1,\ldots,p-1$; see 
Not.\,2.7, and Th.\,2.8, Subsec.\,2.3.}
are trivial for any large $p$.
This result and the well-known relative case indicate
that the probabilities are not uniform
in the following way: 

For small values of odd $n$, the real components
 $\Cl_{\omega^{p-n}}$ are trivial (deep result of [Ku, S])
 and for small values of even $m$, the relative components
 $\Cl_{\omega^{p-m}}$ are trivial (because of the evident
 nondivisibility by $p$ of the first
Bernoulli numbers $B_2, \ldots, B_{m_0}$); so that the real 
components $\Cl_{\omega^{p-3}}, \ldots ,\Cl_{\omega^{p-n_0}}$,
for a small odd $n_0$, and the relative components 
$\Cl_{\omega^{p-2}}, \ldots ,\Cl_{\omega^{p-m_0}}$,
for a small even $m_0$, are trivial,
which implies, by reflection, that the real components 
$\Cl_{\omega^{2}}, \ldots ,\Cl_{\omega^{m_0}}$ are trivial
and the cyclotomic units $\eta_{\omega^{2}}^{},\ldots,\eta_{\omega^{m_0}}^{}$
are not local $p$th powers at $p$.\,\footnote{The equivalence between
$\Cl_{\omega^{p-k}} \ne 1$ and $\eta_{\omega^{k}}$ being a local $p$th power
($k$ even) is given by the theory of $p$-adic $L$-functions or the 
reflection theorem; see Example 2.9.}

In the particular speculative case of the existence of a solution in
the first case of  Fermat${}'$s equation, from  results of
Krasner [Kr], [G2], and many authors,
for small values of odd $n'$, the last
Bernoulli numbers $B_{p-n'}$ must be divisible by $p$, say
$B_{p-3}, \ldots, B_{p-n'_0}$ for a small odd $n'_0$,
giving the nontriviality of the  relative components
$\Cl_{\omega^{3}}, \ldots ,\Cl_{\omega^{n'_0}}$ and the fact that 
the cyclotomic units $\eta_{\omega^{p-3}}^{},\ldots,\eta_{\omega^{p-n'_0}}$
are {\it local} $p$th powers (but not global $p$th powers because of the 
previous result of Soul\'e, at least up to min$\,(n_0, n'_0)$),
which creates a significant defect for the probabilities.

\smallskip
As we  see from the classical literature, strong diophantine
or analytic arguments are absent, even when 
the $p$-rank of the class group is involved since this $p$-rank is used as a 
formal variable. Moreover the second case is rarely studied.

\smallskip
Of course a great part of the point of view developped here is not
really new (many papers of the early twentieth century, contain
overviews of  our point of view) but we 
intend to organize the arguments in a more conceptual and accessible way, mainly to 
avoid Bernoulli${}'$s numbers considerations, 
and to suggest forthcoming studies in a more diophantine
or analytic context by using radicals instead of ideal classes.

\smallskip
 We will see on this occasion that class field theory, in its various aspects,
allows us to find again {\it all} classical technical properties, without dreadful 
computations. 

\smallskip
Some papers already go partially in this direction (e.g. 
Angl\`es [A2, A3], Granville [G1, G2], Helou [He1, He2], 
Terjanian [Te], Thaine [Th1, Th2, Th3], and many others).

Finally, we must mention that all these studies strongly depend on the 
base field (here $\Q$) since it is shown in [A2] that many results or 
conjectures fail for the Fermat equation over a number field 
$k\ne \Q$.

\medskip
In Section 1 we recall some basic facts
for the convenience of the reader; they can also be found for instance in 
Washington${}'$s book [Wa].

\smallskip
In Section 2 we recall some very useful properties of class field 
theory (notion of $p$-primarity which avoids painful computations, 
reflection theorems in the general setting developped in [Gr1, II.5.4])
and we introduce the radical $W$ associated to a solution in any case 
of the Fermat equation. 

Then we explain the insufficiency of the local study of FLT, and we
put the bases of a global approach with $W$
which does not separate the first and second cases of FLT.
We also examine the influence of a solution 
of the Fermat equation on other arithmetic invariants.

\smallskip
In Section 3, for the first case of FLT, we study $p$-adically
the radical $W$, introduced in 
Section 2, and show how Mirimanoff${}'$s polynomials are related to this 
radical, without use of Bernoulli${}'$s numbers; moreover we modify
these polynomials by introducing the characters of the Galois group, 
which illuminates the class field theory context.

From this, we show that
the classical Kummer and Mirimanoff congruences
are directly the expression of reflection theorems.

To be complete, we revisit some $p$-adic studies, as those of  Eichler [E1, E2],
covering works of  Br\"uckner [Br1, Br2] and  Skula [Sk1, Sk2].

We then return to the well-known fact that  Wieferich${}'$s criterion is
a consequence of reciprocity law and, in an Appendix,
we give a proof suggested by Qu\^eme; for
this simpler proof, we interpret, with current technics, some works of Fueter--Takagi (1922) and 
Inkeri (1948) (see [R, IX.4]) which do not use reciprocity law.

Finally we give a standard proof of the Germain--Wendt theorem,
and  intro\-duce some (perhaps new) ideas to compare
Mirimanoff${}'$s polynomials and Gauss${}'$s sums,
and to study ``\,Mirimanoff${}'$s sums\,'' defined as sums of roots 
of unity.

\smallskip
In Section 4, we give some conclusions and prospectives in various 
directions.

\smallskip
We are aware of the futility
of this attempt, but we believe that it can be help\-ful
(or disappointing) for those who wish
to pursue this kind of methodo\-logies.

\section{Classical results depending on a solution of Fermat${}'$s equation}

\medskip
Let $p$ be a prime number, $p > 2$.
Let  $a$, $b$, $c$ in $\Z\Sauf \{0\}$ be pairwise relatively 
prime integers, such that $a^p+b^p+c^p = 0$.
In the second case of FLT, we suppose that $p \,\vert \, c$.

We have the identity:
$$a^p+b^p = (a+b)\, \No_{K/\Q}\ (a + b \,\zeta) = -c^p,$$
where $\zeta$ is a primitive $p$th root of unity, $K = \Q(\zeta)$, and
$\No_{K/\Q}$ is the norm map in $K/\Q$.

\smallskip
Let $\mathfrak p$ be the unique prime ideal $(1 - \zeta)\,\Z[\zeta]$
of $K$ dividing $p$. We have ${\mathfrak p}^{p-1} = p\,\Z[\zeta]$.

\begin{lemm} Let $\nu$ be the $p$-adic valuation of $c$.
If $\nu \geq 1$, then 
$a+b = p^{\nu p -1} c_0^p$ and $\No_{K/\Q}(a + b \,\zeta) = p\, c_1^p$, 
with $p \notdiv c_0\,c_1$ and $p^\nu c_0\,c_1 = -c$. If $\nu = 0$
then $a+b = c_0^p$ and $\No_{K/\Q}(a + b \,\zeta) =  c_1^p$
with $c_0\, c_1 = -c$.\end{lemm}

\begin{proof} If $p \,\vert\, c$, there exists $i$, $0 \leq i \leq p-1$, 
such that $a + b \,\zeta^i \in {\mathfrak p}$; thus $a + b \,\zeta^j \in {\mathfrak p}$
for all $j = 0, \ldots, p-1$ since $a + b \,\zeta^j \equiv a + b 
\,\zeta^i \bmod {\mathfrak p}$ for any $j$.

So $p \,\vert\, a+b$ and, since $p \notdiv b$, the ${\mathfrak p}$-adic 
valuations of $a+b$ and $b\,(\zeta-1)$ are $\mu (p-1)$ for some $\mu 
\geq 1$ and 1, respectively.

Since $p>2$, the ${\mathfrak p}$-adic 
valuation of $a+b\,\zeta = a+b + b\,(\zeta-1)$ is equal to 1
as well as for the conjugates 
$a+b\,\zeta^i$,  $i = 1, \ldots, p-1$.
The ${\mathfrak p}$-valuation of $\No_{K/\Q}(a + b \,\zeta)$ is thus
equal to $p-1$ and that of $a+b$ is $\mu(p-1) = (\nu p-1)(p-1)$,
and the lemma follows. \end{proof}

\begin{lemm} {\sl Let $\ell \ne p$ be a prime number dividing $c$.
Then $\ell \,\vert\, \No_{K/\Q}(a + b \,\zeta)$ if and only if $\ell 
\notdiv a+b$ (i.e., ${\rm g.c.d.}\,(c_0,\, c_1) = 1$).
Any $\ell \,\vert\, \No_{K/\Q}(a + b \,\zeta)$ is totally 
split in $K/\Q$.}
\end{lemm}

\begin{proof}
If $\ell  \,\vert\, \No_{K/\Q}(a + b \,\zeta)$
we may suppose  that $a + b \,\zeta \in \mathfrak l$
for a suitable $\mathfrak l \,\vert\, \ell$ so that 
 $\zeta$ is congruent modulo  $\mathfrak l$ to a rational,
 $\mathfrak l$  is totally split in $K/\Q$, thus $\ell$ is congruent 
to 1 modulo $p$.

The case $\ell \notdiv a+b$ is clear.
If $\ell \,\vert\, a+b$ and if 
${\mathfrak l}\,\vert\, a + b \,\zeta$
for ${\mathfrak l}\,\vert\,\ell$, we get
$b\,(\zeta -1) \in {\mathfrak l}$ (absurd since  $\ell \notdiv b$.).
Thus $\ell \notdiv \No_{K/\Q}(a + b \,\zeta)$.\end{proof}

\begin{coro} {\sl(i)  We have $(a + b \,\zeta)\, \Z[\zeta] = {\mathfrak p} \,
{\mathfrak c}_1^p$ if $p\,\vert\, c$, where ${\mathfrak c}_1$
is an integral ideal prime to ${\mathfrak p}$,  and $(a + b \,\zeta)\,\Z[\zeta] = 
{\mathfrak c}_1^p$ if not. We have $\No_{K/\Q}({\mathfrak c}_1) = c_1$.

(ii) Moreover ${\mathfrak c}_1 = \prod_{\ell\vert c_1}
{\mathfrak l}^{\nu_{\ell}}$, $\nu_\ell > 0$, where 
${\mathfrak l}$ is, for each $\ell \,\vert\, c_1$, a suitable (unique) prime 
ideal above $\ell$.}
\end{coro}

\begin{proof} We have only to prove that if ${\mathfrak l} \,\vert\, a + b 
\,\zeta$,
then for any conjugate ${\mathfrak l}_i$ (by mean of the 
automorphism $\zeta \too \zeta^i$, $i \ne 1$), we have ${\mathfrak l}_i 
\notdiv a + b \,\zeta$; indeed, if not we would have
$b\,(\zeta^{-i}-\zeta) \in {\mathfrak l}$ (absurd).
Thus the ideal $\big(\frac{a+b\,\zeta}{1-\zeta} \big)\, \Z[\zeta]$
or $(a+b\zeta)\, \Z[\zeta]$ is characterized by its norm $c_1^p$ and 
is a $p$th power. \end{proof}

\begin{rema} (i) By permutation we have the following, with 
evident notations:
\begin{eqnarray*}
  &&a+b = p^{\nu p -1}c_0^p \ \ {\rm or} \ c_0^p, 
      \ \   \No_{K/\Q}(a+b \,\zeta) = p\, c_1^p \ \   {\rm or} \ c_1^p,  \ 
      \,{\rm with}\   -c = c_0\, c_1, \\
&&b+c =  a_0^p,  \ \ \  \No_{K/\Q}(b+c \,\zeta) =  a_1^p,  \ \,{\rm 
with} \ -a = a_0\, a_1, \\
 &&c+a = b_0^p,  \ \ \ \No_{K/\Q}(c+a \,\zeta) =  b_1^p,  \ \,{\rm 
 with}\ -b = b_0\, b_1,\\
 &&{\rm g.c.d.}\, (a_0,\, a_1) = {\rm g.c.d.}\,  (b_0,\, b_1) = {\rm 
 g.c.d.}\,  (c_0,\, c_1) = 1,
  \end{eqnarray*}
 \begin{eqnarray*}
&&(a+b\,\zeta)\,\Z[\zeta] = {\mathfrak p}\,{\mathfrak c}_1^p \ \ \hbox{or}\ 
\,{\mathfrak c}_1^p,
\ \ \hbox{with $\No_{K/\Q}({\mathfrak 
c}_1) = c_1$} ,  \\
&&(b+c\,\zeta)\,\Z[\zeta]  = {\mathfrak a}_1^p, \ \ \hbox{with $ \No_{K/\Q}({\mathfrak 
a}_1) = a_1$},   \\
&&(c+a\,\zeta)\,\Z[\zeta] = {\mathfrak b}_1^p, \ \ \hbox{with $ \No_{K/\Q}({\mathfrak 
b}_1) = b_1$}. \\
 \end{eqnarray*}

 \vspace{-0.2cm}
 (ii) All the prime numbers dividing $a_1b_1c_1$ are totally split 
 in $K/\Q$; thus any (positive) divisor of $a_1b_1c_1$ is congruent to 1 modulo $p$. \end{rema}
 
 \smallskip
 These computations and the proofs of FLT in particular cases
 suggest the following conjecture.
 
\begin{conj} {\sl Let $p$ be a prime number, $p>3$, and 
$K=\Q(\zeta)$, where $\zeta$ is a primitive $p$th root of unity.
Put  ${\mathfrak p} := (1 - \zeta)\,\Z[\zeta]$.

\smallskip\noindent
Then for $x,\,y \in \Z \Sauf \{0\}$, with g.c.d. $(x,y)=1$, the equation
$(x+y \,\zeta)\,\Z[\zeta] = {\mathfrak p}\, {\mathfrak z}^p \  {\rm or}\  
{\mathfrak z}^p$
(depending on whether $x+y \equiv 0 \bmod (p)$ or not),
where ${\mathfrak z}$ is an ideal of $K$ prime to~${\mathfrak p}$,
has no solution except the trivial cases:
 $x+y \,\zeta = \pm (1-\zeta) $ and $\pm 
(1+\zeta)$}. \end{conj}

In other words, considering the two relations 
$(a+b\,\zeta)\,\Z[\zeta] = {\mathfrak p} {\mathfrak c}_1^p$ (or ${\mathfrak c}_1^p$)
and  $a+b = p^{\nu p-1} c_0^p$  (or  $c_0^p$),
equivalent to the existence 
of a solution of the Fermat equation, we assert that 
the second is unnecessary, the first 
one being equivalent to $\No(a+b\,\zeta) = p\, c_1^p$  (or  $c_1^p$).
It is likely that this conjecture has already  been stated, 
 but we have found no reference.

\section{Algebraic Kummer theory and reflection theorems}

This Section is valid for the two cases of FLT.

\subsection{$p$-primarity -- local $p$th powers}
The following Theorem 2.2 will be essential to clarify some aspects of 
ramification in Kummer cyclic extensions of degree $p$ of $K$. Let $ K_{\mathfrak p}$ be
the ${\mathfrak p}$-completion of the field $K$ (see [Gr1, I.6.3] 
for the classical notion of $p$-primarity due to Hasse).

\begin{lemm}Let $\alpha \in K^\times$ be prime to $p$ 
and such that $\alpha\,\Z[\zeta]$ is the $p$th power of an ideal of 
$K$.\,\footnote{Such numbers are called {\it pseudo--units} 
since units are a particular case; we will use this word to simplify.}

\noindent
The number $\alpha$ is $p$-primary (i.e., $K(\sqrt [p] {\alpha}\,)/K$ is 
unramified at ${\mathfrak p}$) if and only if it is a local $p$th power (i.e., 
 $\alpha \in K_{\mathfrak p}^{\times p}$).
This happens if and only if $\alpha$ is congruent to a $p$th power
modulo ${\mathfrak p}^{p} = (p) \,{\mathfrak p}$.
\end{lemm}

\begin{proof} One direction is trivial. Suppose that $K(\sqrt [p] 
{\alpha}\,)/K$ is 
unramified at ${\mathfrak p}$; since $\alpha$ is a pseudo-unit,
this extension is unramified as a global 
extension and is contained in the $p$-Hilbert class field $H$ of $K$. The 
Frobenius automorphism of ${\mathfrak p}$ in $H/K$ depends on the 
class of ${\mathfrak p}$ which is trivial since ${\mathfrak p}= 
(1 - \zeta)$; so ${\mathfrak p}$ splits totally in $H/K$, thus in $K(\sqrt [p] {\alpha}\,)/K$, 
proving the first part of the proposition.
The final congruential condition of $p$-primarity is well known (see 
e.g. [Gr1, Ch. I, \S\,6, (b)]).

\smallskip
Warning: the general condition of $p$-primarity in $K$ is 
``\,$\alpha$ congruent to a $p$th power
modulo ${\mathfrak p}^{p} = (p) \,{\mathfrak p}$\,'', but the general 
condition to be a local $p$th power at ${\mathfrak p}$
in $K$ is ``\,$\alpha$ congruent to a $p$th power
modulo ${\mathfrak p}^{p+1} = (p) \,{\mathfrak p}^2$\,''. The fact that 
``\,$\alpha$ is a pseudo-unit of $K$ implies the 
equivalence\,'' is nontrivial and specific of the pseudo-units of the
$p$th cyclotomic field
(such studies are given in [Th3], for special pseudo-units,
by means of explicit polynomial computations).\end{proof}

We have the following consequence, due to Kummer for units, 
which can be generalized to pseudo-units.

\begin{theo}  Every pseudo-unit $\eta$ of $K$,
congruent to a rational (respectively to a $p$th power) modulo 
$p$, is $p$-primary, thus a local $p$th power at ${\mathfrak p}$. If moreover the $p$-class 
group of $K$ is trivial, $\eta$ is a global $p$th power.\end{theo}

\begin{proof} We have, for a suitable rational $\rho$, $\eta^{p-1}
\equiv \rho^{p-1} \equiv 1 \bmod (p)$ in $\Z_{(p)}[\zeta]$,
where $\Z_{(p)}$ is the localization of $\Z$ at $p$.

Put  $\eta^{p-1} = 1 + 
p \,\delta$, $\delta \in \Z_{(p)}[\zeta]$, and $(\eta) = {\mathfrak n}^p$;
taking the norm
of the relation $(\eta^{p-1})={\mathfrak n}^{(p-1)p}$   we get
 $\No_{K/\Q}(\eta^{p-1}) = n^{(p-1)p}$ with $n^{p-1}\equiv 1 \bmod (p)$, hence
$1 \equiv 1 + p \,{\rm Tr}_{K/\Q}(\delta) \bmod (p^2)$ giving 
${\rm Tr}_{K/\Q}(\delta) \equiv 0 \bmod (p)$, thus $\delta \in 
{\mathfrak p}$, proving the first part of the theorem (see Lem.\,2.1).

If $\eta \equiv u^p  \bmod (p)$, $u = \sum u_i\,\zeta^i \in 
\Z_{(p)}[\zeta]$, then $u^p \equiv \sum u_i^p =: \rho \in \Z_{(p)}$
modulo $p$; recipro\-cally, $\eta \equiv \rho  \bmod (p)$
implies $\eta \equiv \rho^p  \bmod (p)$.

The extension
$K(\sqrt[p]\eta\,)$ is thus unramified; so
if the $p$-class group of $K$ is trivial, this extension must be 
trivial, which finishes the proof.\end{proof}

When the $p$-class group of $K$ is trivial,  $K$ is said to be
$p$-regular (in the Kummer sense),
which is here equivalent to its $p$-rationality; this property 
implies in general the above result for units. 
See [MN], [JN], [GJ] for these notions in general, and [AN]
where the Kummer property is generalized. See Subsections 2.5, (a)  
and (b) for the study of the invariants ${\mathcal T}(K)$ and $R_2(K)$ whose 
triviality characterizes the $p$-rationality and the 
$p$-regularity (in the ${\rm K}$-theory sense), respectively.

\subsection{Introduction of some radicals}
We begin by the following remarks, 
from a solution $(a, b, c)$ of the Fermat equation,
which are the key of the present study.
 
\begin{rema}
(i) We  note that we have $(a+b\,\zeta)\Z[\zeta] = {\mathfrak p}{\mathfrak c}_1^p$
or ${\mathfrak c}_1^p$ (see Cor.\,1.3, (i), or Rem.\,1.4, (i)).
This means that the Kummer cyclic extensions (of degree $p$ or 1)
$K(\sqrt [p]{a+b\,\zeta^i}\,)/K , \ \ i=1,\ldots,p-1$,
are $p$-ramified (i.e. unramified outside $p$).
In the same way,
$K(\sqrt [p]{b+c\,\zeta^j}\,)/K,\ \ K(\sqrt [p]{c+a\,\zeta^k}\,)/K$,
$j,\, k=1,\ldots,p-1$,
are $p$-ramified cyclic extensions.

\smallskip
(ii)  When $p\,\vert\, c$, the extensions $K(\sqrt 
[p]{b+c\,\zeta^j}\,)/K$, $j=1,\ldots,p-1$,  are unrami\-fied:
indeed we have $b+c\,\zeta^j  \equiv b \bmod (p)$,
hence the conclusion with Theorem 2.2.

But we know that these extensions must split at ${\mathfrak p}$
which implies that necessarily $c \equiv 0 \mod (p^2)$.\,\footnote{
We have  $b+c\,\zeta
= (b+c)\big (1+ \frac{c}{b+c}\,(\zeta-1)\big)$ where
$b+c = a_0^p$.
Let $1+ \frac{c}{b+c}\,(\zeta-1) = (1+ u\,(\zeta - 1))^p$ locally; if
$u\equiv u_0 \bmod {\mathfrak p}$, with $u_0 \in \Z$, then
$\zeta^{-u_0}\,(1+ u\,(\zeta - 1)) \equiv 1 \bmod {\mathfrak p}^2$,
giving $1+ \frac{c}{b+c}\,(\zeta-1)\equiv 1 \bmod (p)\,{\mathfrak 
p}^2$, thus $c \equiv 0 \bmod (p)\,{\mathfrak  p}$, hence modulo 
$p^2$.}

\smallskip
We have $c+a\,\zeta^k  = \zeta^k\,(a + c\,\zeta^{-k})$ with
$a + c\,\zeta^{-k} \equiv a \bmod (p)$; thus 
in the compositum 
$K(\sqrt [p]{\zeta}\,, \, \sqrt [p]{c+a\,\zeta^k}\,)$
(where  $K(\sqrt [p]{\zeta}\,)/K$ is also $p$-ramified) 
we obtain the  unramified extensions
$K\big (\sqrt [p]{a+c \,\zeta^{k'}}\,\,\big)\big/K$, $k'=1,\ldots,p-1$,
and similarly with $c+b\,\zeta^j$.

\smallskip
(iii) If $p\,\vert\, c$, then from  Corollary 1.3, (i),
the  pseudo-units $\,\Frac{a + b\, 
\zeta^i}{1-\zeta^i}$ are such that 
$\Frac{a + b\, \zeta^i}{1-\zeta^i} = \Frac{a + b}{1-\zeta^i}-b
\equiv -b \bmod (p)$ since $a+b$ is of $p$-valuation
$\nu\,p-1 \geq 2$. Theorem 2.2 implies that the
$\Frac{a + b\, \zeta^i}{1-\zeta^i}$ are local $p$th powers at
${\mathfrak  p}$ and that the extensions
$K\big (\sqrt [p]{\frac{a + b\, \zeta^i}{1-\zeta^i}}\,\big)\big/K$
are unramified.\end{rema}

\begin{nota} Let $E_p$ be the group of $p$-units of $K$.
Then $E_p = \langle \,\zeta,\, 1 - 
\zeta\,\rangle \oplus E^+$,
where $E^+$ is the group of units of the maximal 
real subfield $K^+$ de $K$.
Put $ E^+ = \langle \,\varepsilon_i\,
\rangle_{i=1,\ldots,{\frac{p-3}{2}}}$, and
for $i,j,k = 1,\ldots, p-1$, put:
\begin{eqnarray*}\hspace{0.3cm}
     \Omega_{\ }  &:=& \, \langle\, a+b\,\zeta^i, \,b+c\,\zeta^j,\, c+a\,\zeta^k 
  \,\rangle, \\
  \Gamma_{\ } \, &:=& \, \langle\,\zeta,\ 1 - \zeta,\ \varepsilon_1, 
  \ldots,\varepsilon_{\frac{p-3}{2}},\ 
a+b\,\zeta^i,\  b+c\,\zeta^j,\  c+a\,\zeta^k\, \rangle \,=\, E_p\, \oplus 
\,\Omega, \\
   W_c &:=& \,\langle\, a+b\,\zeta^i\, \rangle_{i} \,.\,
  K^{\times p}/ K^{\times p}, \\
  W_a &:=& \,\langle\, b+c\,\zeta^j\, \rangle_{j} \,.\,
  K^{\times p}/ K^{\times p}, \\
   W_b &:=& \,\langle \, c+a\,\zeta^k\,
    \rangle_{k} \,.\,
  K^{\times p}/ K^{\times p} , \\
  W_{\ } &:=& \,\Gamma  \,.\,  K^{\times p}/ 
  K^{\times  p} .
 \end{eqnarray*}
 If $p\,\vert\,c$ (second case of FLT), we introduce the group:
 
 \smallskip
 $\hspace{1.55cm}\Omega_{\rm prim}\! :=  \langle\, \frac{a+b\,\zeta^i}{1-\zeta^i},
 \,b+c\,\zeta^j,\, a+c\,\zeta^k \,\rangle$,
for which $\Gamma = E_p \oplus \Omega_{\rm prim}$.
\end{nota}
 
\begin{rema} (i) It is easy to see from Corollary 1.3, (ii), that the
$3(p-1) + \frac{p+1}{2}$ elements
$\zeta,\ 
1 - \zeta,\ \varepsilon_1, \ldots,\varepsilon_{\frac{p-3}{2}},\ 
a+b\,\zeta^i,\  b+c\,\zeta^j,\  c+a\,\zeta^k,\  i,j,k=1,\ldots,p-1$,
are  multiplicatively independent and, due to their particular form, the 
idea is that they are largely independent in $K^{\times}/ K^{\times p}$
(this is the main diophantine argument).

\smallskip
Unfortunately, this is probably very difficult to prove since it looks 
like Vandiver${}'$s conjecture (which applies to the cyclotomic 
$p$-units, generated by $1 - \zeta$ and its conjugates, which are not 
independent in $K^{\times}/ K^{\times p}$ as soon as Vandiver${}'$s 
conjecture is false). But in fact we will see 
below that the required condition is not the total independence
of the above numbers in $ K^{\times}/ K^{\times p}$ because of analytic formulas.

\smallskip\smallskip
(ii) It is evident that $\zeta$,
$1 - \zeta$, $\varepsilon_1, \ldots,\varepsilon_{\frac{p-3}{2}}$ are 
independent in $ K^{\times}/ K^{\times p}$ since it is by definition a
$\Z$-basis of $E_p$.

\smallskip
(iii) We have $W = \Gamma\,.\,  K^{\times p}/K^{\times  p}$ and 
$E_p\,.\,  K^{\times p}/K^{\times  p} \simeq E_p/E_p^p$;
then: 
$$\Gamma\,.\,  K^{\times p}/E_p\,.\,  K^{\times p}
\simeq \Gamma/\Gamma \cap (E_p\,.\,  K^{\times p}) \simeq 
\Omega/\Omega \cap (E_p\,.\,  K^{\times p})$$ 
whose order is the degree
$\big [K(\sqrt[p]{\Gamma}\,) \,:\, K(\sqrt[p]{E_p}\,)\big ]$.

\smallskip
(iv) If $p\,\vert\,c$, then $K(\sqrt[p]{\Omega_{\rm prim}}\,)/K$ is 
unramified and $K(\sqrt[p]{\Gamma}\,) / K(\sqrt[p]{E_p}\,)$
is unramified hence ${\mathfrak p}$-split
of degree $(\Omega_{\rm prim}:\Omega_{\rm prim} 
\cap (E_p\,.\,  K^{\times p}))$ (nonramification 
and decomposition propagate by extension),
which will be interpreted in Subsection 2.3.\end{rema}

Denote by $K(\sqrt[p] W\,)$ the extension $K(\sqrt[p] \Gamma\,)$.
We conclude  (Rem.\,2.3) that the extension $K(\sqrt[p] W\,)/K$ is a ${\Pl_p}$-ramified 
$p$-elementary abelian extension of $K$ (i.e., abelian of exponent $p$),
where ${\Pl_p}$ is the set of places of $K$ 
above $p$ (here reduced to the singleton $\{ {\mathfrak p}\}$).

\subsection{Use of class field theory: abelian ${\Pl_p}$-ramification}
Let $H_{\Pl_p}$ be the maximal ${\Pl_p}$-ramified abelian pro-$p$-extension of $K$, 
and let $\Cl_{\Pl_p}$ be the generalised $p$-class 
group of $K$ (i.e., the direct limit of the 
$p$-ray class groups modulo rays groups of conductor a power of $p$);
we have:
$${\rm Gal\,}(H_{\Pl_p}/K) \simeq \Cl_{\Pl_p}.$$

From the general reflection formula proved in [Gr1, II.5.4.1, (iii)] we 
obtain:\,\footnote {For any abelian group $A$
we denote by ${\rm rk}_p(A)$ the $\F_p$-dimension of $A/A^p$.}
$${\rm rk}_p(\Cl_{\Pl_p}) - {\rm rk}_p(\Cl^{\Pl_p}) =
\vert \,\Pl_p\, \vert + p-1 - \hbox{$\Frac{p-1}{2} = 
\Frac{p+1}{2}$}. $$ 

Recall that in this formula, $\Cl^{\Pl_p}$ (the $\Pl_p$-class group) 
is the quotient of the $p$-class group $\Cl$ by the subgroup generated by the classes of 
the prime ideals above $p$, which gives, as we have seen, $\Cl^{\Pl_p} = \Cl$.
 
 \smallskip
 From the above, since $K(\sqrt[p] W\,) \subseteq H_{\Pl_p}$, we get: 
 $$ {\rm rk}_p(\Cl) =  {\rm rk}_p(\Cl_{\Pl_p}) -  \hbox{$\Frac{p+1}{2}$}
 \geq {\rm rk}_p(W) - \hbox{$\Frac{p+1}{2}$}. $$
 Now we can prove the following from a solution $(a,b,c)$ of the 
 Fermat equation:
 
 \begin{theo} Let $W$ be the radical generated,
in  $K^{\times}/ K^{\times  p}$, by the group of $p$-units $E_p$ and the numbers
$a+b\,\zeta^i$, $b+c\,\zeta^j$, $c+a\,\zeta^k$, $i, j, k 
=1,\ldots,p-1$.\,\footnote{In the second case of FLT with $p\,\vert\,c$, $a+b\,\zeta$
is not a pseudo-unit, but $\frac{a+b\,\zeta}{1-\zeta}$, $b+c\,\zeta$, $c+a\,\zeta$ are pseudo-units;
thus $W$ is generated by $1-\zeta$ and pseudo-units.}

\smallskip\noindent
Then we have the inequalities
$\ {\rm rk}_p(W)  \leq \frac{p+1}{2} + {\rm rk}_p(\Cl) \leq p$.

\smallskip\noindent
If moreover $p$ is regular (i.e., if $\Cl$ is trivial) then
$\ W = E_p/E_p^p$.\end{theo}

\begin{proof} From many authors 
(see e.g. [G3] for more history), we know that the relative class 
number $h^-$, i.e., the order of the relative class group 
$C^- \!:= {\rm Ker}\big( \No_{K/K^+} : C \too C^+ := C_{K^+} \big)$,
is such that ${\rm log}(h^-) <  \hbox{$\frac{p}{4}$} {\rm log}(p)$
which proves that ${\rm rk}_p(\Cl^-) \leq  \hbox{$\frac{p-1}{4}$}$.
From classical Hecke--Leopoldt reflection theorem, we get
${\rm rk}_p(\Cl^+) \leq   {\rm rk}_p(\Cl^-)$
giving the (very bad) inequality
${\rm rk}_p(\Cl) \leq \hbox{$\frac{p-1}{2}$}$,
and the first part of the theorem. 

If $p$ is regular we get ${\rm rk}_p(W) \leq \frac{p+1}{2}$;
since $W$ contains $E_p/E_p^p$ which is of $p$-rank $\frac{p+1}{2}$
we have the equality, proving the theorem.\end{proof}

In the regular case we obtain the following (see Not.\,2.4):

\smallskip
(i) {\it First case of FLT}. From Remark 2.5, (iii), we obtain
$\Omega \subset E\,.\,K^{\times p}$ since in the first case the 
elements of $\Omega$ are pseudo-units. Then in that case, all
the elements 
$a+b\,\zeta^i$, $b+c\,\zeta^j$, and $\,c+a\,\zeta^k$ are of the form
$\varepsilon \,.\, \alpha^p$, $\varepsilon \in E$, $\alpha \in \Z[\zeta]$.
Of course, one can take for $\varepsilon$ a cyclotomic unit
since the group of cyclotomic units is of prime to $p$ index in $E$.

\smallskip
(ii) {\it Second case of FLT}. From Remark 2.5, (iv),
and Theorem 2.2,  we obtain
$\Omega_{\rm prim} \subset K^{\times p}$; so in the 
second case (with $p\,\vert\, c$), all the elements
$\frac{a+b\,\zeta^i}{1-\zeta^i},\,b+c\,\zeta^j\,$, and $a+c\,\zeta^k$
are global $p$th powers,which can perhaps simplify the usual proof.

\smallskip\smallskip
From this we obtain easily the classical proofs by Kummer of FLT
as those given in [W, Th.\,1.1 and Th.\,9.3] or in [Hel, Chap. 1, \S\,8.4].

\smallskip
However, Eichler${}'$s theorem [E1, E2] (i.e., ${\rm rk}_p(\Cl^-) \leq [\sqrt 
{p+1} - 1.5\,]$ implies the first case of FLT), that we will discuss and prove later 
(Th.\,3.14), may be considered 
as a wide generalization of the regular case, but limited to the first case of FLT
(see also [W, Th.\,6.23] or [R, IX.7] for similar proofs).

\smallskip\smallskip
In the general case, the unlikely equality ${\rm rk}_p(\Cl^+) = {\rm 
rk}_p(\Cl^-)$ used for the proof of Theorem 2.6 supposes the 
following facts (see [Gr1, II.5.4.9.2]) for which we introduce the 
characters of the Galois group:

\begin{nota} (i) Let $g={\rm Gal\,}(K/\Q)$ and
let $\omega$ be the character of Teichm\"uller of $g$ (i.e.,
the character with values in $\mu_{p-1}(\Q_p)$ such that 
for the $s_k \in g$ defined by $s_k(\zeta) = \zeta^k$, 
$k=1,\ldots,p-1$, $\omega(s_k)$ is the unique $(p-1)$th root
of unity in $\Q_p$, congruent to $k$ modulo $p$).
We will also write $\omega(k) := \omega(s_k)$.

\smallskip
(ii) Any irreducible $p$-adic character of $g$ is of the form
$\chi := \omega^m$, for $m \in \{1, \ldots, p-1\}$; we denote by $\chi_0$ the 
unit character ($m= p-1$).

\smallskip
If $\chi$ is any $p$-adic character of $g$, we put $\chi^* := \omega \chi^{-1}$
(reflection character).

\smallskip
(iii) The idempotent corresponding to $\chi$ is:
\vspace{-0.1cm}
$$ e_\chi :=  \hbox{$\frac{1}{p-1}$} \sm_{s \in g} \chi(s^{-1})\, s = 
\hbox{$\frac{1}{p-1}$} \sm_{k=1}^{p-1} \chi^{-1}(k)\, s_k\,\in \Z_p[g] . $$
 The action of $e_\chi$ on a $\Z_p[g]$-module is well-defined; for
a $\Z[g]$-module $M$, we use instead
the $\Z_p[g]$-module $M\tensorZ \Z_p$ or the $\Z_p[g]$-module
$M\tensorZ \F_p \simeq M/M^p$; by abuse of notation we  write $M_\chi := M^{e_\chi}$ 
for the $\chi$-component of $M$ in the above sense. 

For instance, we denote by ${\rm rk}_p(\Cl_{\chi})$ the $p$-rank of the
$\chi$-component $\Cl_{\chi}$ of the $p$-class group $\Cl$ \big(\,$\Cl_{\chi}$ is 
thus the maximal submodule of $\Cl$ on which $g$ acts  via $c^s = 
c^{\chi(s)}$ for all $s \in g$ and any class $c \in \Cl_{\chi}$\big).

For the group $E$ of units, $E_\chi := E^{ e_\chi}$ must be 
interpreted in $E\tensorZ \Z_p$ or $E/E^p$ depending on the context.

(iv) Let $K_\chi$ be the subfield of $K$ fixed by ${\rm Ker}(\chi)$.
\end{nota}

\smallskip
 To be self-contained, we recall here the main 
classical results which will be of constant use.

\begin{theo}[Prerequisites]
{\rm (i) (Kummer duality; see [Gr1, Rem.\,II.5.4.3])}. Let $H{\st [p]}$ be the 
$p$-elementary $p$-Hilbert class field of $K$, $A:={\rm Gal}(H{\st [p]}/K)$, 
and $R$ the radical of $H{\st [p]}$ (i.e., $A\simeq \Cl/\Cl^p$ and
$H{\st [p]} = K(\sqrt[p] R\,)$).

For any character 
$\chi$ of $g$ and for $\chi^* := \omega\,\chi^{-1}$ we have the canonical 
isomorphism of $g$-modules:
$${\rm Gal} (K(\sqrt[p] {R_{\chi^*}}\,)/K) \simeq A_\chi\,\cdot$$

\smallskip
Then we have $R_{\chi^*} \subset K_{\chi^*}$ and $K(\sqrt[p] 
{R_{\chi^*}}\,)/K$ splits over $K_{\chi}\,\cdot$

\medskip
{\rm (ii) (Reflection theorems; see [Gr1, 5.4.9.2, ``\,Analysis of a 
result of Hecke\,''])}. For any even character $\chi \ne \chi_0$
and for $\chi^* := \omega\,\chi^{-1}$ we have:
$${\rm rk}_p((Y/Y_{\rm prim})_{\chi^*}) =
{\rm rk}_p(\Cl_{\chi^*}) -  {\rm rk}_p(\Cl_{\chi})=
1 - {\rm rk}_p((Y/Y_{\rm prim})_{\chi}), $$

\smallskip
where $Y$ is the group of  pseudo-units of $K$ (elements 
equal to the $p$th power of an ideal prime to ${\mathfrak p}$),
and where $Y_{\rm prim}$ is the subgroup of $p$-primary pseudo-units 
(i.e., local $p$th powers at ${\mathfrak p}$).

\medskip
{\rm (iii) (Main theorem on cyclotomic fields of Thaine--Ribet--Mazur--Wiles--Kolyvagin;
see [W, \S\,15.4])}.
 For any even character $\chi \ne \chi_0$
 and for $\chi^* := \omega\,\chi^{-1}$ we have:

\smallskip
$\ \ \ \ \ \bullet\ \ $ $\vert\,\Cl_\chi\,\vert = 
\vert\,\big(\,  \langle \,\varepsilon_\chi\,\rangle :
 \langle \,\eta_\chi\,\rangle\,\big) \,\vert_p^{-1}$, where
 $\varepsilon_\chi$ is a generator of $E_\chi$ and $\eta_\chi=(1-\zeta)^{e_\chi}$.
 
$\ \ \ \ \ \bullet\ \ $ $\vert\,\Cl_{\chi^*}\,\vert = 
\vert\, b_{\chi^*} \, \vert_p^{-1}$, where
$b_{\chi^*} := \hbox{$\frac{1}{p}$} \sm_{k=1}^{p-1} ({\chi^*})^{-1} 
(k)\, k$.
\end{theo}

The use of the deep result (iii) is not really necessary in this paper but 
it clarifies the reasonings since we are only interested by the 
logical aspects of the influence of a solution of Fermat${}'$s
equation on these invariants and not by an optimization of the statements.

\begin{exa}{\rm  If for an  even $\chi \ne \chi_0$, 
the group $\Cl_{\chi^*}$ is nontrivial, there exists a 
nontrivial ${\chi^*}$-pseudo-unit $\alpha_{\chi^*}$ (i.e.,
$\alpha_{\chi^*} \notin K^{\times p}$).

 \medskip
If $\alpha_{\chi^*}$ is $p$-primary then from (i) this defines a 
$\chi$-unramified cyclic extension of degree $p$ of 
$K_\chi$; so that  $\Cl_{\chi} \ne 1$ and
$\big(\,  \langle \,\varepsilon_\chi\,\rangle :
 \langle \,\eta_\chi\,\rangle\,\big) \equiv 0 \bmod (p)$
from  (iii) (counterexample to the Vandiver conjecture).
 
 \medskip
 If $\alpha_{\chi^*}$ is not $p$-primary then from (ii) we get
 ${\rm rk}_p((Y/Y_{\rm prim})_{\chi^*}) = 1$ and
 ${\rm rk}_p((Y/Y_{\rm prim})_{\chi}) = 0$ which implies that all
 the $\chi$-pseudo-units are $p$-primary, especially 
 $\varepsilon_\chi$,
 hence $\eta_\chi \in  \langle \,\varepsilon_\chi\,\rangle$ is also a 
 local $p$th power at ${\mathfrak p}$. We have 
 obtained a class field theory version of a result given by the 
following properties of $p$-adic $L$-functions: 
 \begin{eqnarray*}
 L_p(0,\chi) &\equiv& L_p(1,\chi) \bmod (p) \ \ \hbox{[W, 
 Cor.\,5.13]}\,, \\
 L_p(0,\chi) &=& - b_{\chi^*} \ \ \hbox{[W, Th.\,5.11]}\,, \\
  L_p(1,\chi) &=&  \hbox{$\frac{\tau(\chi)}{p}$} 
  \sm_{k=1}^{p-1} \chi^{-1} (k) {\rm log}(1-\zeta^k)
  =  \hbox{$\frac{\tau(\chi)}{p}$} \, {\rm log} (\eta_{\chi}^{p-1}) 
  \ \ \hbox{[W, Th.\,5.18]}\,, 
 \end{eqnarray*} 
  
where the Gauss sum $\tau(\chi)$ is of ${\mathfrak p}$-valuation
 $\leq p-2$, giving easily $b_{\chi^*} \equiv 0$
 $\bmod\, (p)$ if and only if $\eta_{\chi}$ is a local $p$th power 
 at ${\mathfrak p}$ (see 
 Subsec. 3.3 and 3.4)}.
\end{exa}

Then from the above, concerning the
equality ${\rm rk}_p(\Cl^+) = {\rm rk}_p(\Cl^-)$,
we would have, for each
even  $\chi$ such that ${\rm rk}_p(\Cl_{\chi^*}) \geq 1$, the alternative
${\rm rk}_p(\Cl_{\chi^*}) \geq 
2$, or ${\rm rk}_p(\Cl_{\chi^*}) = 1$ and in the writing ${}_p\Cl_{\chi^*} = 
\langle\,\cl ({\mathfrak a}_{\chi^*}) \,\rangle$ then
${\mathfrak a}_{\chi^*}^p =: (\alpha)$ with
 $\alpha$  $p$-primary; all this is of course very strong because of the 
proba\-bilistic value of ${\rm rk}_p(\Cl^+)$ discussed in ``\,Introduction
and Generalities\,''.

\smallskip
We will return to reflection theorem in the proof of Theorems 3.7 and 3.9.

\medskip
If we refer to [W, \S\,6.5], the value of ${\rm rk}_p(\Cl)$
is conjecturally ${\mathcal O} \big(\frac{{\rm log}(p)}{{\rm log}({\rm 
log}(p))} \big)$.
With such a result, the inequality of Theorem 2.6 would be:
$${\rm rk}_p(W) \leq \hbox{$\Frac{p+1}{2}$} + {\mathcal O} \Big( 
\hbox{$\Frac{{\rm log}(p)}{{\rm log}({\rm log}(p))}$} \Big),  $$
noting that the principal term $\frac{p+1}{2}$ comes from the $p$-units;
this means, from Remark 2.5, (iii), that most of the 
elements of $\Omega$ (see Not.\,2.4) are of the 
form $\varepsilon \,.\, \alpha^p$, $\varepsilon\in E_p$, $\alpha\in \Z[\zeta]$.
In case Vandiver${}'$s conjecture is satisfied, Theorem 2.6 reduces to:
 $${\rm rk}_p(W) \leq \Frac{p+1}{2} + \Frac{p-1}{4},
 \ \,{\rm instead\  of} \ \, \leq  p. $$

It is implausible that the $p$-rank of the radical $W$, generated by
the images in $K^{\times}/ K^{\times p}$ of the
$3(p-1) + \frac{p+1}{2}$
multiplicatively independent elements of $\Gamma$, could be less than~$p$.

\subsection{Comparison of the local and global approaches}
Now we intend to show that any restriction to the
local case  leads to
the following fact, where $K_{\mathfrak p}$ is the completion of $K$ 
at ${\mathfrak p}$:
$${\rm rk}_p \Big( {\rm Gal\,}
\big( K_{\mathfrak p}(\sqrt[p] W\,) / K_{\mathfrak p} )\Big) \leq p ; $$
in other words, the four radicals $W_a$, $W_b$, $W_c$, $E_p/E_p^p$
become largely dependent by ${\mathfrak p}$-completion of the base field.

More precisely, we have $K_{\mathfrak p}(\sqrt[p] W\,) =
K_{\mathfrak p}(\sqrt[p]{ W_{\mathfrak p}}\,)$,
where $W_{\mathfrak p} = \Gamma\,.\, K_{\mathfrak p}^{\times p}/
K_{\mathfrak p}^{\times p}$
is the local radical generated by the image  in 
$K_{\mathfrak p}^{\times}/ K_{\mathfrak p}^{\times p}$  of the
$3(p-1) + \frac{p+1}{2}$ elements
$\zeta,\,1 - \zeta ,\, \varepsilon_1,\, 
\ldots,\varepsilon_{\frac{p-3}{2}}$,  
$a+b\,\zeta^i,\, b+c\,\zeta^j,\, c+a\,\zeta^k$, $i,j,k=1,\ldots,p-1. $

\smallskip
For instance, if $p\,\vert\, c$, $W_{\mathfrak p}$ is the local 
radical generated by $E_p$ (see Rem.\,2.5, (iv)).

\medskip
Since ${\mathfrak p}$ splits completely in 
$H$ and is totally ramified in $H_{\Pl_p}/H$, 
by local class field theory the $p$-rank of
${\rm Gal\,}\big(H_{\Pl_p}/H\big)$
is less than or equal to the  $p$-rank of the
inertia group of the maximal ${\mathfrak p}$-ramified  abelian 
pro-$p$-extension $M_{\mathfrak p}$ of $K_{\mathfrak p} = H_{\mathfrak p}$,
equal to the $p$-rank of the
subgroup of units  of $K_{\mathfrak p}^\times$, thus equal to $p$.

Since $K_{\mathfrak p}(\sqrt[p]{ W_{\mathfrak p}}\,) 
= H_{\mathfrak p} (\sqrt[p]{ W_{\mathfrak p}}\,)  \subseteq M_{\mathfrak 
p}$, this yields as expected:
$${\rm rk}_p ( W_{\mathfrak p} ) = {\rm rk}_p \Big( {\rm Gal\,}
\big(K_{\mathfrak p} (\sqrt[p]{ W_{\mathfrak p}}\, ) / K_{\mathfrak p} 
\big)\Big)\leq p.$$

Returning to the global situation and using Theorem 2.6, we obtain 
directly that:
$${\rm rk}_p( W_{\mathfrak p} ) \leq {\rm rk}_p (W)
\leq   \hbox{$\Frac{p+1}{2}$} + {\rm rk}_p(\Cl) \leq p,$$

\noindent
which is surprising since the global inequality
is obtained via an approximate analytic formula.

So in the local situation we only have the following informations:
$${\rm rk}_p( W_{\mathfrak p} ) \leq  \Frac{p+1}{2} + {\rm rk}_p(\Cl), $$
knowing that (in a ``numerical'' point of view) $W_{\mathfrak p}$ does not 
contain more than $p$ independent elements in  $K_{\mathfrak p}^{\times}/ 
K_{\mathfrak p}^{\times p}$, to be compared with the global situation:
$${\rm rk}_p(W) \leq  \hbox{$\Frac{p+1}{2}$} + {\rm rk}_p(\Cl), $$
knowing that the $p$-rank of $W$ in  $K^{\times}/ 
K^{\times p}$ is only limited by $3(p-1) + \frac{p+1}{2}$.

\medskip
In the two directions (local or global), a contradiction (i.e., a proof 
of FLT) would be obtained by proving the following inequalities:

\smallskip
(i) {\it In the local case}:
$$ {\rm rk}_p(W_{\mathfrak p}) > \hbox{$\Frac{p+1}{2}$}
+ {\rm rk}_p(\Cl) ,$$
under the fact that ${\rm rk}_p(W_{\mathfrak p})$ is 
$p - \delta(p)$, where the defect $\delta(p)$, in the first case of 
FLT, depends essentially of the local properties of
Mirimanoff${}'$s polynomials (see Th.\,3.5 and Th.\,3.9), which gives
the sufficient condition to be proved:
$$\delta(p) < p - \hbox{$\Frac{p+1}{2}$} - {\rm rk}_p(\Cl) =
   \hbox{$\Frac{p-1}{2}$} - {\rm rk}_p(\Cl) ,$$
which is unusable with the analytic inequality ${\rm rk}_p(\Cl) \leq 
\frac{p-1}{2}$ equivalent to $\delta(p) = 0$.\footnote{Note that Mirimanoff${}'$s 
congruences tend to yield a large $\delta(p)$.}

In the second case of FLT, such a proof is also impossible since, as we have 
seen, ${\rm rk}_p(W_{\mathfrak p}) \leq {\rm rk}_p(E_p) = \hbox{$\frac{p+1}{2}$}$.

\smallskip
(ii) {\it In the global case, for the two cases of FLT}:
$${\rm rk}_p (W) >\hbox{$\Frac{p+1}{2}$} + {\rm rk}_p(\Cl)  ,$$
under the fact that ${\rm rk}_p(W)$ is 
$3(p-1) + \hbox{$\frac{p+1}{2}$} - \Delta(p)$,
where the defect $\Delta(p)$ depends on deep 
diophantine properties, which gives the sufficient condition to be proved:
$$ \Delta(p) < 3(p-1) + \hbox{$\Frac{p+1}{2}$}- \hbox{$\Frac{p+1}{2}$}
- {\rm rk}_p(\Cl)   = 3\,(p-1) - {\rm rk}_p(\Cl)  ,$$
realized as soon as $\Delta(p) < 5\, \frac{p-1}{2}$
with the analytic inequality ${\rm rk}_p(\Cl) \leq \frac{p-1}{2}$,
which may be provable.

\begin{rema} (i)
In the previous analysis, one may object that in an evident way, 
global radicals and class groups give equivalent informations (in spite 
of the fact that here we consider generalized classes), but we 
insist on the fact that these radicals, hence the corresponding classes, 
are of a very special nature (see 
for instance Conjecture 1.5, specific of this particular case).

\smallskip
(ii) If we replace the fundamental units $\varepsilon_i$ by the
cyclotomic units, we obtain the radical
$\wt W = \langle\, \zeta, 1 - \zeta^n, a+b\,\zeta^i, b+c\,\zeta^j, 
c+a\,\zeta^k\,\rangle\,.\, K^{\times p}/ K^{\times 
p}$, $\ n, i,j,k=1,\ldots,p-1$,
all the elements being of the special form $x + y\,\zeta^q$.

\smallskip
The radical $\wt W$ is of $p$-rank 
$3(p-1) + \hbox{$\frac{p+1}{2}$} - \wt \Delta(p)$, which
requires to prove that
$\wt \Delta(p) < 3(p-1)-{\rm rk}_p(\Cl)$, with $\wt \Delta(p) \geq \Delta(p)$
because of possible cyclotomic units being $p$th powers of units (defect of Vandiver${}'$s 
conjecture), which seems to be acceptable, even if $\wt \Delta(p)$ is not 
so good, to perform $\wt \Delta(p) < 5\,\frac{p-1}{2}$. \end{rema}

\subsection{Links with other invariants}
Since analytic aspects are important to get good upper bounds, it is 
useful to connect (or replace) the classical class group with other 
invariants. Moreover, a solution of Fermat${}'$s equation has important consequences
on any  arithmetic invariant, as the following ones.

\medskip
(a) {\it Case of the torsion subgroup of ${\rm Gal\,} (H_{\Pl_p}/K)$.}

\smallskip
 Recall that ${\rm Gal\,} (H_{\Pl_p}/K) \simeq
\Cl_{\Pl_p}$  is isomorphic to $\Z_p^{\frac{p+1}{2}}  \oplus {\mathcal T}$,
where ${\mathcal T}$ is the (finite) $p$-torsion subgroup. Thus we get
${\rm rk}_p(W) \leq {\rm rk}_p(\Cl_{\Pl_p}) = 
\hbox{$\frac{p+1}{2}$} + {\rm rk}_p({\mathcal T})$,
giving
${\rm rk}_p({\mathcal T}) \geq  {\rm rk}_p(W) - 
\hbox{$\Frac{p+1}{2}$}$.

If ${\mathcal G}$ is the Galois group of the maximal ${\Pl_p}$-ramified 
pro-$p$-extension of $K$, then the group  ${\mathcal G}$ is defined by  $d$
generators and $r$ relations, where:
\begin{eqnarray*}
   d &=&  {\rm rk}_p({\rm H}^1({\mathcal G}, \Z/p\Z)) = {\rm 
   rk}_p(\Cl_{\Pl_p}) =  \hbox{$\Frac{p+1}{2}$} +{\rm rk}_p({\mathcal T})  , \\
  r &=&  {\rm rk}_p({\rm H}^2({\mathcal G}, \Z/p\Z)),
 \end{eqnarray*}
with the duality ${\rm H}^2({\mathcal G}, \Z/p\Z)^* \simeq {}_p{\mathcal T}$ (see for 
instance [Gr1, App., Th.\,2.2]), giving:
$$ {\rm rk}_p({\rm H}^2({\mathcal G}, \Z/p\Z)) \geq  {\rm rk}_p(W) - 
\hbox{$\Frac{p+1}{2}$} = 3(p-1) - \Delta(p) . $$

One may expect that there
exist some constraints on such cohomology groups.

\smallskip
The field $K$ is said to be $p$-rationnal (see [MN]) if ${\mathcal T}=1$, which is 
equivalent to $\Cl = 1$ ($K$ is $p$-regular in the Kummer sense).

\smallskip
From the reflection theorem (see [Gr2, Th.\,10.10]), we have 
for any $\chi$ with $\chi^* = \omega\,\chi^{-1}$:\,\footnote{For a direct proof,
use the fact that the relative 
component  $\Cl^-_{\Pl_p}$ is the sum of ${\mathcal T}^-$ and of the 
Galois group of the compositum of the
relative $\Z_p$-extensions giving the representation $\Z_p[g]^-$;
the real part $\Cl^+_{\Pl_p}$ is the sum of ${\mathcal T}^+$ and of 
$\Z_p$ with trivial character; so [Gr1, Th.\,II.5.4.5] gives the formula.}
$${\rm rk}_p({\mathcal T}_\chi) = {\rm  rk}_p(\Cl_{\chi^*}). $$

From
the interpretation of the reflection principle for the groups 
$\Cl_{\chi}$ recalled in the Theorem 2.8, (ii)
(see also [Gr1, II.5.4.9.2]), we obtain a similar result between the groups
${\mathcal T}_\chi$ and ${\mathcal T}_{\chi^*}$: 
$${\rm rk}_p((Y/Y_{\rm prim})_{\chi^*}) =
 {\rm rk}_p({\mathcal T}_\chi) - {\rm  rk}_p({\mathcal T}_{\chi^*}) = 
1- {\rm rk}_p((Y/Y_{\rm prim})_{\chi}), $$ for any  even $\chi$, where
 $Y$ is the group of pseudo-units
and $Y_{\rm prim}$ the subgroup of $p$-primary pseudo-units.

\smallskip
Hence, for the group ${\mathcal T}$, the ``\,Vandiver
conjecture\,'' is ${\mathcal T}^{-} = 1$.

\medskip
Let us mention the two relations (equalities up to a $p$-adic unit):
$$\vert\,{\mathcal T}^+\,\vert = \vert\,\Cl^+\,\vert\,.\,{\Frac{\rm 
Reg^+}{\rm Disc^+}}\,\virg\ \ \vert\,{\mathcal T}^-\,\vert =
\Frac{\vert\,\Cl^-\,\vert}{\big(\Z_p \,{\rm log}(I^-) : 
\Z_p \,{\rm log}(P^-) \big)}\,\virg $$
where Reg$^+$ is the $p$-adic regulator, Disc$^+$ the discriminant,
of $K^+$, $I$ the group of ideals prime to $p$, $P$ the subgroup of 
principal ideals, of $K$;
if $c$ is the complex conjugation and 
${\mathfrak a}$ an ideal of $K$,  let $n$ be such that 
${\mathfrak a}^{n\frac{1-c}{2}} = (\alpha)$, then
${\rm log}({\mathfrak a}^{\frac{1-c}{2}}) := \frac{1}{n}{\rm 
log}(\alpha)$ where ${\rm log}$ is the Iwasawa logarithm
for which ${\rm log}(p) = 0$
(note that for the minus part, the units do not enter in the use of ${\rm 
log}$; see [Gr1, Cor.\,III.2.6.1, Rem.\,III.2.6.5] for more details and 
references).

\medskip
 As for the class group, the existence of a solution of Fermat${}'$s
 equation has a great influence on the group ${\mathcal T}$, for 
 instance on the study of the index ${\big(\Z_p \,{\rm log}(I^-) : 
\Z_p \,{\rm log}(P^-) \big)}$ regarding the relations $(x+y\,\zeta) = 
{\mathfrak p} \,{\mathfrak z}_1^p$ or ${\mathfrak z}_1^p$
giving:
$${\rm log}\Big({\mathfrak z}_1^{\frac{1-c}{2}}\Big) := 
\Frac{1-c}{2}\,\Frac{1}{p}{\rm log}(x+y\,\zeta) =
\Frac{1-c}{2}\,\Frac{1}{p}{\rm 
log}\Big(1+\Frac{y}{x+y}\,(\zeta-1)\Big). $$

\medskip
Mention also the following reasoning giving another interpretation of 
a result of Iwasawa [Iw], which
may have some interest\,\footnote{From a talk given
in 1982 in the University Laval, Qu\'ebec; published in the 
mathematical series, N$^{\rm o} 20$ (1984), of the department of mathematics.}:

\smallskip
For an even ${\chi}$, since $\Z_p\,{\rm log}(P^-) = {\rm log}(U^-)$
where $U$ is the group of principal units of $K_{\mathfrak p}$,
we obtain easily:
$$|{\mathcal T}_{\chi^*}| = \frac {|\Cl_{\chi^*}|}
{\big(e_{\chi^*}\,.\,{\Z_p}{\rm log}(I):
e_{\chi^*}\,.\,{\rm log}(U)\big)}\,\cdot$$

The main theorem on cyclotomic fields (see Th.\,2.8, (iii)) gives
$|\Cl_{\chi^*}| = |b_{\chi^*}|_p^{-1}$ (the $p$-part of the 
corresponding generalized Bernoulli number $b_{\chi^*} \in \Z_p$).

\smallskip
We know that for any prime ideal ${\mathfrak l}$ of $K$, ${\mathfrak 
l}\ne {\mathfrak p}$, 
we have:
$${\mathfrak l}^{\,p S} = {\mathcal G}({\mathfrak l})^p\, \Z[\zeta], $$
where $S := \frac{1}{p} \sum_{k=1}^{p-1} k\,s_k^{-1}$
is the Stickelberger element\,\footnote{We have 
${e_{\chi^*}}\,.\,S = b_{\chi^*}\,.\,e_{\chi^*}$;
this explains that we use a different definition from that of [W]
for the generalized Bernoulli numbers.} and ${\mathcal G}({\mathfrak l})$
the Gauss sum:
$${\mathcal G}({\mathfrak l}) := - \sm_{t \in F_{\mathfrak l}} 
\psi(t)\,\zeta_{\ell}^{{\rm tr}(t)}, $$
where $F_{\mathfrak l}$ is the residue field, $\psi$ the canonical 
character of order $p$ of $F_{\mathfrak l}^\times$,
$\zeta_{\ell}$ a primitive $\ell$th root of unity,
and ${\rm tr}$
the trace in the residual extension $F_{\mathfrak l}/\F_\ell$.
Thus taking ${\rm log}$ we obtain for all even $\chi$:
$$e_{\chi^*}\,.\, S \,.\, {\rm log} ({\mathfrak l}) =
e_{\chi^*} \,.\, b_{\chi^*} \,.\,{\rm log}({\mathfrak l}) =
    e_{\chi^*} \,.\, {\rm log}({\mathcal G}({\mathfrak l}) ). $$
    
Then  $|b_{\chi^*}|_p^{-1} \, 
e_{\chi^*}\,.\,{\Z_p}{\rm log}({\mathfrak l}) = e_{\chi^*}\,.\,
{\Z_p} {\rm log}({\mathcal G}({\mathfrak l}))$, thus:
$$|{\mathcal T}_{\chi^*}| = \frac {|b_{\chi^*}|_p^{-1}}
{\Big( \frac{1}{|b_{\chi^*}|_p^{-1}} \, 
e_{\chi^*}\,.\,{\Z_p}{\rm log}\,({\mathcal G}) : e_{\chi^*}\,.\,{\rm 
log}\,(U)\Big)}\, \virg$$
where ${\mathcal G}$ is the group generated by all the Gauss sums ${\mathcal G}({\mathfrak l})$.

\smallskip
So, the Vandiver conjecture for $\chi$ even 
($\Cl_{\chi} = {\mathcal T}_{\chi^*} = 1$) is equivalent to 
the fact that $e_{\chi^*}\,.\,{\Z_p}\,{\rm log}\,({\mathcal G}) =  e_{\chi^*}\,.\,{\rm 
log}(U)$, and the whole Vandiver conjecture is equivalent to 
the fact that the images of the Gauss sums in
$U$ generate the minus part of this $\Z_p$-module.

\medskip
(b) {\it Case of the regular and wild kernels.}

\medskip
Recall the fundamental diagram of ${\rm K}$-theory, in which ${\rm W\!K}_2(K)$
is called the wild kernel and $ {\rm R}_2(K)$ the regular 
kernel in the ordinary sense. We specify the diagram recalled in [Gr1, II.7.6] 
to the case of the cyclotomic field $K$
($h$ is the Hilbert symbol and $h^{\rm reg}$ the regular Hilbert 
 symbol, which is explicit):
%
 %
 $$\begin{array}{ccccccccc} 1 &\too & {\rm W\!K}_2(K) &\tooo & {\rm K}_2(K) &
\stackrel{h}{\tooo} & \plus_{v\in\Pl_{}^{\rm nc}}\mu(K_v) &
\stackrel{\pi}{\tooo} & \mu(K)  \too  1 \\
 & &   \Big \vert \Big \vert   & &
\Big \vert \Big \vert & &\Big \downarrow  
{\raise -3pt \hbox{$\hspace{-2.04mm}\downarrow$}} {\st\ } & &
\hspace{-0.7cm} \Big \downarrow   {\st \ } \vspace{0.2cm}\\
1 & \too &  {\rm R}_2(K)  & \tooo &  {\rm K}_2(K) &
\stackrel{h^{\rm reg}}{\tooo} &
\ \ \,\plus_{v\in\Pl_{}^{\rm nc}} \mu(K_v)^{\rm reg} &\tooo & 
\hspace{-0.7cm} 1 \,, \\
\end{array}  $$

since $({\rm R}_2: {\rm W\!K}_2) = 1$ for $K = \Q(\zeta)$
(use [Gr1, II.7.6.1]).

\medskip
 For ${\rm R}_2$ we have a Kummer interpretation,
coming from  results of Tate  [Ta], which is given
by the exact sequence:
$$1 \too \mu_p \otimes  
{\rm N}_2 \tooo \mu_p\otimes  W_{\Pl_p} \mathop{\tooo}^f
{}_{p}\hspace{-1.2pt}{\rm R}_2 \too 1, $$
where $W_{\Pl_p}$ is the initial radical of $H_{\Pl_p}/K$, $f$ being
defined by  $f(\zeta \otimes \alpha) := \{\zeta\,,\,\alpha\} $
for all  $\alpha\in W_{\Pl_p}$, and where
${\rm N}_2 := \{ \alpha\in K^\times,\ \{\zeta\,,\,\alpha \} = 1 \}/ K^{\times p}$
(Tate${}'$s kernel) is such that (as $g$-modules):
$$\mu_p \otimes {\rm N}_2 \simeq (\mu_p \otimes \mu_p) \oplus 
\mu_p^{\frac{p-1}{2}}. $$

\smallskip
We then have ${\rm rk}_p({\rm R}_2) = {\rm rk}_p( W_{\Pl_p})
- \frac{p+1}{2} = {\rm rk}_p(\Cl)$
(see [Gr1, II.7.7.2.2]).
More precisely, using characters, we have here another principle of reflection, since 
we must associate $\chi$ with $\ov \chi := \omega^{-1} \chi = 
(\chi^*)^{-1}$, giving for all~$\chi$:
$${\rm rk}_p({\rm R}_{2,\,\chi}) = {\rm  rk}_p(\Cl_{\omega^{-1} \chi})=
{\rm  rk}_p({\mathcal T}_{{\omega^2 \chi}^{-1}}). $$

As for the group ${\mathcal T}$, we get for any  even $\chi$:
 $${\rm rk}_p((Y/Y_{\rm prim})_{\chi^*}) =
 {\rm  rk}_p({\rm R}_{2 ,\,\omega^2\chi^{-1}}) -
  {\rm rk}_p({\rm R}_{2 ,\,\omega\chi}) = 
1 - {\rm rk}_p((Y/Y_{\rm prim})_{\chi}), $$ and
 ``\,Vandiver${}'$s conjecture\,'' for $R_2$  is $R_2^{-} = 1$.

This can be deduced from the above exact sequence 
 by proving that the groups 
$\langle \,\zeta\,\rangle \tensorZp \Cl$ and ${}_p{\rm R}_2$ are isomorphic
 $g$-modules, which is coherent with the above reflection. Another 
proof uses the isomorphism proved by  Jaulent [J] between 
${\rm W\!K}_2/ ({\rm W\!K}_2)^p$ and
$\langle \,\zeta\,\rangle \tensorZp \wt \Cl$, where
 $\wt {\Cl}$ is the logarithmic $p$-class group,
 and the isomorphism  $\wt {\Cl} \simeq \Cl$ for $K$ (see [Gr1, Exer. III.7.1]).
 
 \smallskip
 The field $K$ is said to be $p$-regular (in the ${\rm K}$-theory 
 sense) if the $p$-Sylow of the 
 regular kernel $R_2$ is trivial (see [JN, GJ]); here it is the case if and only if $\Cl=1$.
 
\smallskip
We have here a complete parallelism between regular kernel and class 
group (with another Galois action), which may be interesting by studying for instance the map $f$ 
on the elements $x+y\,\zeta$ of the radical $W \subseteq  W_{\Pl_p}$, and so on.

\smallskip
We know that ${\rm R}^+_2 := {\rm R}_2(K^+)$ is given by the value at $-1$ of the 
Dedekind zeta function $\zeta^{}_{K^+}$ of $K^+$; more precisely, after the proof of 
Birch--Tate conjecture by  Wiles (on this subject, see e.g. Greither [Gre]) we get:
$$|\,{\rm R}_2^+ \,| = \hbox{$\Frac{24\, p}{2^{\frac{p-1}{2}}}$} \, \, 
\big|\,\zeta^{}_{K^+}(-1)\,\big| $$
(see Washington${}'$s book [Wa, Ch.\,IV] to compute the analytic  expression
of $|\,{\rm R}_2^+ \,|$). For the minus part $|\,{\rm R}_2^- \,|$, we don$'$t know 
convenient analytic formula as for $|\,{\mathcal T}^-\,|$;
we only have the isomorphism
${}_p{\rm R}_2^- \simeq (\langle \,\zeta\,\rangle \tensorZp \Cl)^-$.

\section{Some classical local considerations revisited (first case of 
FLT)}

To study the $p$-rank of the radical $W$ we begin with the partial 
radical $W_c$, in the first case of FLT, or the radical generated by
$W_c$ and the units.

\medskip
Thus in this Section we 
suppose that $p\notdiv c$; so we will have similar results by 
permutations of $\{ a,\,b,\,c \}$ with no more global informations
as explained in Section 2; moreover, since
$a+ b \,\zeta = \zeta\, (b + a \,\zeta^{-1})$,
the radical $W$ contains the conjugates of $b + a \,\zeta^{-1}$ and 
we can add  the transpositions of the set $\{ a,\,b,\,c \}$, so 
that the reasonings (in the first case of 
FLT) are valid for any $(x,\,y) \in 
\{(a,\,b),\, (b,\,a),\, (b,\,c),\, (c,\,b),\, (c,\,a),\, (a,\,c) \}$.

\subsection{Logarithmic derivative: Mirimanoff${\,}'\!$s polynomials}
We need, {\it once for all}, a convenient characterization of $p$-primarity;
the best way is to use the method of derivation of Eichler. 
Everything depends on this.

\medskip
From a solution $(a,b,c)$ in the first case of Fermat${}'$s equation, we 
study the relation:
\zero
$$\prd_{i=1}^{p-1} (a+b\,\zeta^i)^{\lambda_i} = \alpha^{\,p}, \ \, \lambda_i\in 
\{0,\ldots, p-1\}, \ \alpha \in \Z[\zeta].\ \ntm$$
\initnt
\ntt{Without any change, we can study the 
same relation in $\Z_p[\zeta]$ instead of $\Z[\zeta]$; in that case, 
we will obtain a N.S.C. (see Th.\,3.5).}

Since $a+b$ is a $p$th power (Lem.\,1.1), it is equivalent to consider, 
for $e := \frac{b}{a+b}$:
$$\prd_{i=1}^{p-1} (1+e \,(\zeta^i-1))^{\lambda_i} = \beta^{\,p}, \ \, \lambda_i\in 
\{0,\ldots, p-1\}, \ \beta \in \Z_{(p)}[\zeta]. $$
This relation is equivalent to the polynomial relation:
$$F(X) := \prd_{i=1}^{p-1} (1+e \,(X^i-1))^{\lambda_i} = G(X)^p + 
A(X) \, \Phi_p(X),\ \,G,\,A \in \Z_{(p)}[X], $$
where $\Phi_p(X)$ is the $p$th cyclotomic polynomial.

\begin{lemm} We can choose $G(X)$ modulo $\Phi_p(X)$ such that:
$$F(X) = \prd_{i=1}^{p-1} (1+e \,(X^i-1))^{\lambda_i} = H(X)^p + 
B(X)\,(X^p - 1),\ \,H,\,B \in \Z_{(p)}[X]. $$\end{lemm}

\begin{proof}Since $F(1) = 1$, $G(1)^p + A(1)\, p = 1$, thus 
$G(1)^p \equiv G(1) \equiv 1 \bmod (p)$ and  $G(1) = 1 + \Lambda\, p$,
$\Lambda \in \Z_{(p)}$.
Put $G_1(X) := G(X) - \Lambda \,\Phi_p(X)$; this yields to
$G_1(1) = G(1) - \Lambda\, p = 1$.

We have $F(X) = G_1(X)^p + A_1(X) \,\Phi_p(X)$ for some $A_1(X)$.

We obtain
$F(1) = 1 = G_1(1)^p + A_1(1)\, p = 1 + A_1(1)\, p$, in other words $A_1(1) = 0$.
Thus $A_1(X) = (X-1)\, B(X)$. We then put $H(X) :=G_1(X)$.\end{proof}

By logarithmic derivation, since  $e \not\equiv 0 \bmod (p)$ in the first case of 
FLT and since $F(X)$ is invertible modulo $(p, X^p-1)$, this gives:
$$ \hspace{2.8cm} \sm_{i=1}^{p-1}\Frac{\lambda_i \, i\, X^{i-1}}{1+e\,(X^i-1)} \in (p, 
X^p-1)_{\Z_{(p)}[[X]]}.  \hspace{2.8cm}(1)$$

\begin{rema} From this formula we deduce (taking $X=1$) the 
necessary condition  $\ \sum_{i=1}^{p-1} \lambda_i\,i \equiv 0 \bmod (p)$,
which gives one nontrivial relation between the~$\lambda_i$.
This relation is due to an obstruction on the $\omega$-component
(see Rem.\,3.4 ).

The interest of Lemma 3.1 is that $(p,X^p-1)' \subseteq (p,X^p-1)$.
\end{rema}

The series $\Frac{1}{1+e\,(X^i-1)} = \sm_{j\geq 0} (-1)^j\, e^j\, (X^i-1)^j$
are convergent for the $(X-1)$-adic topology and, since $(X^i-1)^p \in (p, 
X^p-1) = (p, (X-1)^p)$, we obtain, after multiplication by $X$,
the equivalent condition:
$$ \sm_{i=1}^{p-1} \lambda_i \, i\, X^{i}\, \sm_{j=0}^{p-1}(-1)^j\, e^j\, (X^i-1)^j \in (p, (X-1)^p). $$
Thus, using $(X^i-1)^j = \sm_{k\geq 0}(-1)^{j-k}\, \big(^j_k\big)\, X^{ik}, \ \ 
\hbox{\,with $\big(^j_k\big) = 0\,$ for $k> j$}$,
this yields:
$$\sm_{k\geq 0} \,\, \sm_{i=1}^{p-1} \lambda_i \, i\, X^{i(k+1)}\,.\, (-1)^k\,  \sm_{j= 
k}^{p-1}\big(^j_k\big)\, e^j \in  (p, (X-1)^p). $$
Since $j \leq p-1$ and $\big(^j_k\big) = 0$ for $k> j$, we can limit 
$k$ to the value $p-1$; for $k=p-1$ we get the term
$\ \sm_{i=1}^{p-1} \lambda_i \, i\, X^{ip} \,e^{p-1} \equiv 
\sm_{i=1}^{p-1} \lambda_i\,i \equiv 0 \bmod (p, X^p-1)$.

Then, under the condition $\sum_{i=1}^{p-1} \lambda_i\,i \equiv 0 \bmod 
(p)$, we can suppose that $k$ varies from 0 to $p-2$.
Put:
$$\varphi_{k+1} (X) := \sm_{i=1}^{p-1} \lambda_i\,  i\, X^{i(k+1)} ,\ \ 
k=0,\ldots, p-2.$$
We obtain the following condition (2), equivalent to (1) under the 
condition $\sum_{i=1}^{p-1} \lambda_i\,i \equiv 0 \bmod (p)$:
$$\hspace{1.0cm}\sm_{k= 0}^{p-2} \varphi_{k+1} (X)\,.\, A_k \in  (p, (X-1)^p), \, \ {\rm 
with}\ A_k := (-1)^k  \sm_{j=k}^{p-1} \big(^j_k\big)\, e^j. 
\hspace{1.0cm}(2) $$

\begin{lemm} We have $A_k \equiv (-1)^k \,\frac{e^k}{k !}\, D^{(k)}(e)
\equiv \big(\frac{-b}{a}\big)^k \bmod (p)$, $k=0,\ldots,p-2$,
where $D(Y) := 1+Y+\ldots+Y^{p-1}$.\end{lemm}

\begin{proof} We have  $A_0 = D(e) = \frac{e^p-1}{e-1} \equiv 1 \bmod 
(p)$ since $e\not\equiv 1 \bmod (p)$ (otherwise $a\equiv 0 \bmod (p)$).
The first general relation giving $A_k$ is immediate by induction, using
$\big(^j_k\big)= \frac{j!}{k!\,(j-k)!}\,$ and
$\,D^{(k)}(Y) = \frac{k!}{0!} +  \frac{(k+1)!}{1!}\,Y + \ldots + 
\frac{(k+p-1-k)!}{(p-1-k)!} \, Y^{p-1-k}$.

\smallskip\smallskip
Since $D(Y) = 1+Y+\ldots+ Y^{p-1} = \frac{Y^p-1}{Y-1}
\equiv (Y-1)^{p-1} \bmod p\,\Z[Y]$, we have 
$D^{(k)} (e) \equiv  (p-1)\ldots(p-k)\,.\,(e-1)^{p-1-k}
 \equiv (-1)^{k}\, k!\, (e-1)^{-k} \bmod (p)$.

Then $A_{k} \equiv (-1)^{k}\, \frac{e^{k} }{k!} \,
D^{(k)} (e) \equiv \big (\frac{e}{e-1}\big)^{k}\!\!\!
= \big (\frac{-b}{a}\big)^{k} \!\bmod (p)$,
hence the result.\end{proof}

We intend to use this formula
in the case of the action of the 
idempotents $e_\chi \in \Z_p[g]$, $\chi = \omega^m$ (where $g = {\rm 
Gal\,}(K/\Q)$) on the previous pseudo-unit $1+e\,(\zeta-1)$
(see Not.\,2.7).

\medskip
 The formulation of the condition $F(X) = H(X)^p + 
B(X)(X^p - 1)$ corresponds to the 
 choice $\lambda_i \equiv \frac{1}{p-1}\omega^{-m}(i)$ modulo $p\,\Z_p[\zeta]$;
 the necessary condition
 $\sum_{i=1}^{p-1} \lambda_i\,i \equiv 0 \bmod (p)$ (see Rem.\,3.2) is satisfied for 
 any $m \in \{ 1,\ldots, p-1\}$, except $m=1$ (i.e., $\chi = \omega$).
 
 For $m \ne 1$ we obtain from the  above:
 $$\varphi_{k+1} (\zeta) = \hbox{$\frac{1}{p-1}$}
 \sm_{i=1}^{p-1} \omega^{-m}(i)\, i\, \zeta^{i(k+1)}
 \equiv \hbox{$\frac{1}{p-1}$}
 \sm_{i=1}^{p-1} \omega^{1-m}(i) \,\zeta^{i(k+1)}\bmod (p) ,$$
$$\sm_{k= 0}^{p-2} \varphi_{k+1} (\zeta)\,.\, A_k =
\sm_{k= 1}^{p-1} \varphi_{k} (\zeta)\,.\, A_{k-1} \equiv 
\hbox{$\frac{1}{p-1}$}\sm_{k= 1}^{p-1} \Big(\sm_{i=1}^{p-1} \omega^{1-m}(i)\, \zeta^{ik}\Big)\,.\,
A_{k-1} \bmod (p).$$

We have obtained the necessary condition (put $j := i\,k$ modulo $p$):
$$-\sm_{k= 1}^{p-1}\sm_{ j= 1}^{p-1} \omega^{1-m}(j k^{-1})\, \zeta^{j}\,.\, A_{k-1}
    = \Big( \sm_{k= 1}^{p-1} \omega^{m-1}( k) \,.\, A_{k-1}\Big)
\Big( - \sm_{j=1}^{p-1}\omega^{1-m}(j)\, \zeta^{j}\Big) \equiv 0 \bmod (p),$$
where:
$$ -\sm_{j=1}^{p-1} \omega^{1-m}(j)\, \zeta^{j} =: \tau(\omega^{1-m}), $$
is the Gauss sum of $\omega^{1-m}$, for which:
$$\tau (\omega^{1-m})\,.\,\ov\tau (\omega^{m-1}) = p,\ \ {\rm where}\ \ 
\ov\tau (\varphi) := -\sm_{k=1}^{p-1}\varphi(k) \zeta^{-k}
= \varphi(-1)\,\tau (\varphi)$$
for any character $\varphi$.

\smallskip
But $\tau(\omega^{1-m})$, as 
element of $\Z_p[\zeta]$, is of ${\mathfrak p}$-valuation
$m-1$,  $\,m \in\{1,\ldots,p-1\}$ (see Prop.\,3.17 in Subsec. 3.4).
The final necessary condition is thus, for $m\ne 1$:
$$\hspace{3.3cm} \sm_{k= 1}^{p-1} \omega^{m-1}( k) \,.\, A_{k-1} \equiv 0 \bmod 
p\,\Z_p[\zeta].\hspace{3.1cm} (3)$$

\begin{rema} \label{remarque6} For $m=1$ (i.e., $\chi = \omega$) a direct computation gives:
 \begin{eqnarray*}
     (1+e \,(\zeta -1))^{e_\omega} &=& \prd_{i=1}^{p-1} 
     \big(1+e \, (\zeta^i-1) \big)^{\frac{\omega^{-1}(i)}{p-1}} 
 \equiv  \prd_{i=1}^{p-1}  \big (1+   e \,\omega^{-1}(i)\,
(\zeta^i-1) \big)^{\frac{1}{p-1}} \\
  &\equiv&  \prd_{i=1}^{p-1}  \big (1+ e\, (\zeta-1) \big)^{\frac{1}{p-1}}
 \equiv   1+ e\, (\zeta-1) \bmod {\mathfrak p}^2,
 \end{eqnarray*}
using $\zeta^i -1 = \frac{\zeta^i -1}{\zeta -1}\,(\zeta -1)
\equiv i\,(\zeta -1) \bmod {\mathfrak p}^2$.\end{rema}
     
\begin{theo} Let $(a,b,c)$ be a solution in the first case of 
Fermat${\,}'\!$s equation; put $e = \frac{b}{a+b}$.
Let $\chi = \omega^m$ be a $p$-adic character of $g$ distinct from $\omega$.

\smallskip\noindent
Then the pseudo-unit $(a+b\,\zeta)^{e_\chi}$ or
$(1+e \,(\zeta-1))^{e_\chi}$ is a  $p$th power in $K_{\mathfrak p}$
if and only if: 
$$\sm_{k= 1}^{p-1} \omega^{m-1}(k)  
\hbox{$\big(\frac{-b}{a}\big)^{k}$} \equiv 0 \bmod (p). $$ \end{theo}

\begin{proof} We have to prove the sufficiency of the condition.
We note that this congruential condition is (for $\chi \ne \omega$)
only equivalent to $F'(X) \in (p,\Phi_p(X)) = (p,(X-1)^{p-1})$
in $\Z_p[X]$, since $(X-1)\Phi_p(X) =X^p-1
\equiv (X-1)^{p} \bmod (p)$ (see (1), (2), (3)).

\smallskip
Suppose that  the condition
$F'(X) \in (p,(X-1)^{p-1})$ is satisfied
for the coef\-ficients $\lambda_i = \frac{1}{p-1}\chi^{-1}(i)$
in  $F(X) = \prod_{i=1}^{p-1} (1+e\,(X^i-1))^{\lambda_i}$. 

\smallskip
Write $F(X) = \sum_{n=0}^{p-1} u_n\, (X-1)^n + U(X)\,(X-1)^{p}$ in 
$\Z_p[X]$; since:
$$F'(X) = \sm_{n=1}^{p-1}n\, u_n (X-1)^{n-1}\! + 
p\,U(X)\,(X-1)^{p-1} + U'(X)\,(X-1)^{p}$$
is in $(p,(X-1)^{p-1})$, 
this yields to $u_n  \equiv 0 \bmod (p)$ for $n=1,\ldots, p-1$.

\smallskip
Then $F(\zeta) \equiv u_0 \bmod (p)$;
from Theorem 2.2, $F(\zeta)$ being a pseudo-unit congruent to 
a rational modulo $p$ is a local $p$th power.
 Which proves the theorem obtained by Thaine [Th3] using 
generalized binomial computations.\end{proof}

\begin{rema} (i) We have obtained that in our viewpoint using radicals, the $p$-primarity
of the pseudo-unit $(1+e \,(\zeta-1))^{e_\chi}$, $\chi 
\ne \omega$, is directly characterized by means of the polynomial:
$$M_m(Z) := \sm_{k= 1}^{p-1} \omega^{m-1}( k) Z^k. $$
As the reader can see, this polynomial 
is a variant of the classical polynomial of Mirimanoff
$\wt M_m(Z) := \sm_{k= 1}^{p-1} k^{m-1} Z^k$ and
is congruent modulo $p$ to it (see [R, VIII.1] for more information;
see [A1] for the use of  Mirimanoff${}'$s polynomials in 
Iwasawa theory over $K$; see [Th2, I] for the definition of
polynomial congruences equivalent to Mirimanoff${}'$s congruences
and giving a direct proof of some Wieferich${}'$s criteria). 

\smallskip\smallskip
(ii) We see that $ M_m(Z)$ comes from the Gauss sum
$\tau(\omega^{m-1})$ that we have encountered before (put $Z = \zeta$),
and this has probably a deep signification (see Subsec. 3.4 for 
some insights).

\smallskip
This shows that this indexation is not convenient; we observe that
$M_m(Z)$ must be denoted $M_{\chi^*}(Z)$, where $\chi^* = \omega \chi^{-1} = 
\omega^{1-m}$, and more generally $M_\varphi(Z) :=
\sum_{k= 1}^{p-1} \varphi^{-1}(k) \,Z^k$ for any character $\varphi$. 

\smallskip
Thus, to summarize:
\vspace{-0.2cm}
 \begin{eqnarray*}
 M_{\chi^*}(Z) &:=& \sm_{k= 1}^{p-1}(\chi^*)^{-1} (k) \,Z^k \\
M_{\chi^*}(\zeta) &=& \sm_{k= 1}^{p-1}(\chi^*)^{-1} (k) \,\zeta^k
= - \tau((\chi^*)^{-1}); 
\end{eqnarray*}
for convenience, we will use the two notations, the rule being
$M_{\omega^h} = M_{p-h}$.

\medskip
We see also that by all
permutations of $a$, $b$, $c$, the $p$-primarity
of the corres\-ponding pseudo-units $(x + y\,\zeta)^{e_\chi}$, $\chi\ne 
\omega$ (i.e., $\chi^*\ne \chi_0$),
is equivalent to the congruence:
$$M_{\chi^*} \hbox{$\Big(\Frac{-y}{x}\Big)$} = \sm_{k= 1}^{p-1}(\chi^*)^{-1}(k)  
\hbox{$\Big(\Frac{-y}{x}\Big)^{k}$} \equiv 0 \bmod (p). $$
This notation which associates $\chi$ (for $(x + y\,\zeta)^{e_\chi}$)
and $\chi^*$ (for $M_{\chi^*} \big(\frac{-y}{x}\big)$)
anticipates the use of reflection theorems.

\smallskip
(iii) The  advantage of this definition of Mirimanoff${}'$s polynomials,
indexed by the characters of $g$, is that they may be related to
characters of some subfields of $K$, giving a more precise
information (use Th.\,2.8, (i)), and the knowledge of
the  $p$-class groups of the subfields may have suitable consequences for
the properties of these polynomials (e.g. $\chi = 
\omega^{\frac{p-1}{2}}$, $\chi^* = \omega^{\frac{p+1}{2}}$).\end{rema}

\begin{theo}[algebraic form of Kummer${}'$s congruences] Let $(a,b,c)$
 be a solution in the first case of Fermat${\,}'\!$s equation.

\smallskip
If for an odd character $\chi \ne \omega$, $M_{\chi^*}\big(\frac{-b}{a} 
\big) \not\equiv 0 \bmod (p)$ (where $\chi^* = \omega \chi^{-1}$), then  the 
$\chi$-component $\Cl_\chi:= \Cl^{e_\chi}$ of the $p$-class 
group is nontrivial.\end{theo}

\begin{proof}
We have $(1+e\,(\zeta-1))^{e_\chi} 
\notin K^{\times p}$ since this pseudo-unit is not a local $p$th power at~${\mathfrak p}$.
Put $(1+e \,(\zeta-1))^{e_\chi}\,\Z_{(p)}[\zeta] = {\mathfrak z}^p$; if the ideal ${\mathfrak z}$
is principal, say ${\mathfrak z} = (z)$, then:  
$$(1+e \,(\zeta-1))^{e_\chi} = \varepsilon\, z^p, \ \, {\rm where}\  \varepsilon \in 
E_\chi := E^{e_\chi} ; $$
since  $\chi$ is odd and distinct from $\omega$ (the character of 
$\langle\,\zeta\,\rangle$), $\varepsilon = 1$, giving a global
$p$th power for $(1+e \,(\zeta-1))^{e_\chi}$ (contradiction). Thus 
$\cl({\mathfrak z}) \in \Cl_\chi$ 
 is nontrivial.\end{proof}

From the main theorem on cyclotomic fields (see Th.\,2.8, (iii))), the 
$p$-valuation of $\vert\,\Cl_\chi\,\vert$ is that
of the genera\-lized Bernoulli number: 
$$b_{\chi} := \hbox{$\frac{1}{p}$} \sm_{k=1}^{p-1} \chi^{-1} (k) k; $$
so $b_{\chi} \equiv 0 \bmod (p)$ or,
equivalentely since $\chi = \omega^m$,
$m \ne 1$ odd, the ordinary  Bernoulli number $B_{p-m}$
is congruent to 0 modulo $p$ (see [W, Cor.\,5.15]). 

\smallskip
Actually  Stickelberger${}'$s theorem is 
sufficient to get $b_{\chi} \equiv 0 \bmod (p)$; if we want the 
reciprocal of  ``\,Herbrand${}'$s theorem\,'', we can use [Ri], [Th4] to 
get that $b_{\chi} \equiv 0 \bmod (p)$ is equivalent to $\Cl_\chi 
\ne 1$. 
 
 \smallskip
 We find again in a more precise way the classical situation
 of Kummer${}'$s congruences which are:
 $$b_{\chi} \,\cdot\,M_{\chi^*} \hbox{$\big(\frac{-b}{a} \big)\equiv 0 \bmod 
 (p)$}. $$
  If $M_{\chi^*}\big(\frac{-b}{a} \big) 
 \not\equiv 0 \bmod (p)$, then  $\Cl_\chi \ne 1$ and
  for $\chi^*$ (even and nontrivial), we know by reflection
  (see Exa.\,2.9) that
 the $\chi^*$-cyclotomic unit $\eta_{\chi^*}:= (1 - \zeta)^{e_{\chi^*}}$
 is a {\it local} $p$th power at $\mathfrak p$.
 It is a global $p$th power if and only if $\Cl_{\chi^*}$ is nontrivial
 (Vandiver${}'$s conjecture false at $\chi^*$).
 
\begin{rema} 
 If $\chi \ne \chi_0$ is even, if $M_{\chi^*}\big(\frac{-b}{a} 
\big) \not\equiv 0 \bmod (p)$, and if the ideal ${\mathfrak z}$
is principal (in the wtiting  $(1+e \,(\zeta-1))^{e_\chi}\,\Z[\zeta] = 
{\mathfrak z}^p$), we only obtain the relation
$(1+e \,(\zeta-1))^{e_\chi} = \varepsilon_\chi \, z^p$,
 where  $\varepsilon_\chi \in 
E_\chi$ is not a local $p$th power at $\mathfrak p$.

The basic example for this is $\Cl_{\chi^*} = 1$, thus 
$\Cl_{\chi} = 1$ (Vandiver${}'$s conjecture true at $\chi$); we then have 
$b_{\chi^*} \not\equiv 0 \bmod (p)$ thus 
$M_{\chi} \hbox{$\big(\frac{-b}{a}\big)$}\equiv 0 \bmod (p)$, 
which implies that $(1+e \,(\zeta-1))^{e_{\chi^*}}$ is a global $p$th power
since $\Cl_{\chi} = 1$.

If $\mathfrak z$ is nonprincipal, then $\Cl_{\chi} \ne 1$
(counterexample to Vandiver${}'$s conjecture), $\Cl_{\chi^*} \ne 1$,
$b_{\chi^*} \equiv 0$ $\bmod\, (p)$, and 
$M_{\chi} \hbox{$\big(\frac{-b}{a}\big)$}$ is a priori arbitrary
(see Rem.\,3.11 for improvements of these reasonings).\end{rema}

\begin{theo}[algebraic form of Mirimanoff${}'$s  congruences: 
the reflection theorem]
    Let $\chi \ne \chi_0$ be even, and let  $\chi^* = 
\omega \chi^{-1}$ ($\chi^*$ is odd distinct from $\omega$).

\smallskip
Then we have $M_{\chi^*}\big(\frac{-y}{x} 
\big)\,.\,M_{\chi}\big(\frac{-y}{x} \big) \equiv 0 \bmod (p)$
for any of the six pairs $(x,y)$ corresponding to a solution in the 
first case of Fermat${}'$s equation.\end{theo}

\begin{proof} To prove this congruence, we suppose that both
$M_{\chi^*}\big(\frac{-y}{x} \big)$ and
$M_{\chi}\big(\frac{-y}{x} \big)$ are not congruent to 0 modulo $p$ to 
obtain a contradiction.

\smallskip
From the Theorem 2.8, (ii), or [Gr1, II.5.4.9.2], the analysis of the reflection 
theorem in $K$ leads to the following equalities ($\chi$ even):
$${\rm rk}_p((Y/Y_{\rm prim})_{\chi^*}) =
{\rm rk}_p(\Cl_{\chi^*}) -  {\rm rk}_p(\Cl_{\chi})=
1 - {\rm rk}_p((Y/Y_{\rm prim})_{\chi}), $$

where $Y$ is the group of  pseudo-units of $K$,
and where $Y_{\rm prim}$ is the subgroup of $p$-primary pseudo-units.

\smallskip
The condition $M_{\chi}\big(\frac{-y}{x} \big) \not\equiv 0 \bmod (p)$ 
is thus equivalent to $(x+y\,\zeta)^{e_{\chi^*}} \in Y\Sauf Y_{\rm prim}$
giving ${\rm rk}_p((Y/Y_{\rm prim})_{\chi^*}) = 1$, and similarly
the condition $M_{\chi^*}\big(\frac{-y}{x} \big) \not\equiv 0$ $\bmod 
\,(p)$ is
equivalent to $(x+y\,\zeta)^{e_{\chi}} \in Y\Sauf Y_{\rm prim}$, 
giving ${\rm rk}_p((Y/Y_{\rm prim})_{\chi}) = 1$ (contradiction).\end{proof}

\begin{coro}  Let $\chi \ne \chi_0$
(i.e., $\chi^* \ne \omega$) be even. Suppose that
 $\Cl_{\chi}$ is trivial (Vandiver${}'$s conjecture true at $\chi$).
 
 \smallskip\noindent
 If $M_{\chi}\big(\frac{-y}{x} 
\big) \not\equiv 0 \bmod (p)$, then $M_{\chi^*}\big(\frac{-y}{x} 
\big) \equiv 0 \bmod (p)$ and the 
fundamental ${\chi}$-unit $\varepsilon_{\chi}$ is $p$-primary
as well as the $\chi$-cyclotomic unit $\eta_{\chi} := (1-\zeta)^{e_\chi}$.\end{coro}
 
\begin{proof} From  $M_{\chi}\big(\frac{-y}{x} 
\big) \not\equiv 0 \bmod (p)$ and Theorem 3.7
we get that $\Cl_{\chi^*}$ is 
of $p$-rank $\geq$~$1$, hence equal to 1 since $\Cl_{\chi}=1$;
so by Kummer duality (see Th.\,2.8, (i)), the radical of the 
corresponding unramified ${\chi^*}$-extension of $K$ is given by the 
fundamental ${\chi}$-unit $\varepsilon_{\chi}$ which is thus 
$p$-primary.
By hypothesis, $E_\chi$ is also generated by
the $\chi$-cyclotomic unit $\eta_{\chi}$.
This is the result obtained in [Th2, II] via  congruential computations.\end{proof}

So, Mirimanoff${}'$s congruences, obtained by ugly computations, are nothing 
but the reflection principle in class field theory.

\begin{rema} Let $\chi$ be even distinct from $\chi_0$.

\smallskip
 (i) If $M_{\chi}\big(\frac{-y}{x} \big) \not\equiv 0 \bmod (p)$, then 
 from the proof of Theorem 3.9 we have
${\rm rk}_p((Y/Y_{\rm prim})_{\chi^*}) = 1$,
${\rm rk}_p((Y/Y_{\rm prim})_{\chi}) = 0$
(all the $\chi$-pseudo-units 
are $p$-primary, especially $\varepsilon_\chi$), and for the class 
group we get
${\rm rk}_p(\Cl_{\chi}) + 1 =  {\rm rk}_p(\Cl_{\chi^*})$, which 
means that the $\chi^*$-class group is nontrivial.
Then $(x+y\,\zeta)^{e_{\chi}}$ is $p$-primary
(which is coherent with $M_{\chi^*}\big(\frac{-y}{x} \big) \equiv 0 \bmod (p)$)
but can be a global $p$th power.

\smallskip
(ii) If $M_{\chi^*}\big(\frac{-y}{x} \big) \not\equiv 0 \bmod (p)$,
${\rm rk}_p((Y/Y_{\rm prim})_{\chi}) = 1$, ${\rm rk}_p((Y/Y_{\rm 
prim})_{\chi^*}) =0$  (all the $\chi^*$-pseudo-units 
are $p$-primary), and
${\rm rk}_p(\Cl_{\chi^*}) = {\rm rk}_p(\Cl_{\chi})$.

\smallskip
(iii) If $M_{\chi^*}\big(\frac{-y}{x} \big) \equiv
M_{\chi}\big(\frac{-y}{x} \big) \equiv 0 \bmod (p)$, then
 $(x+y\,\zeta)^{e_{\chi}}$ and 
$(x+y\,\zeta)^{e_{\chi^*}}$ are $p$-primary, but we dont know if they are 
global $p$th powers or not; if for instance $(x+y\,\zeta)^{e_{\chi}} = z^p$
then the ideal ${\mathfrak c}_1^{e_{\chi}}$ is principal. If
$(x+y\,\zeta)^{e_{\chi}}$ is not of the form $\varepsilon_\chi\, z^p$,
${\mathfrak c}_1^{e_{\chi}}$ is not principal (the $\chi$-class group is nontrivial), and
 $(x+y\,\zeta)^{e_{\chi}}$ defines the 
 radical of a $\chi^*$-unramified extension of $K$ (the $\chi^*$-class 
 group is of course nontrivial).
 
 If $(x+y\,\zeta)^{e_{\chi^*}}$ is not a $p$th power, 
 ${\mathfrak c}_1^{e_{\chi^*}}$ is nonprincipal (because $E_{\chi^*} = 
 1$) and defines the radical of a $\chi$-unramified extension of $K$, 
 giving $\Cl_{\chi} \ne 1$ (Vandiver${}'$s conjecture false at $\chi$), hence also
 $\Cl_{\chi^*} \ne 1$.
 
 \smallskip
(iv) If  $\Cl_{\chi^*}$ is trivial,
then the unit $\varepsilon_\chi$ is not $p$-primary and all the 
$\chi^*$-pseudo-units are $p$-primary (hence global $p$th powers); then
we get $M_{\chi}\big(\frac{-y}{x} \big) \equiv 0 \bmod (p)$.

\smallskip
(v)  For $\chi = \chi_0$, we know that  $\Cl_{\chi^*}=\Cl_\omega$ is
trivial; in this case, $M_{\chi_0}\big(\frac{-y}{x} \big) = 
\sum_{k= 1}^{p-1} \big(\frac{-y}{x} \big)^k$
takes always the value 0 for $\frac{-y}{x}\not\equiv 1\bmod (p)$.

\smallskip\smallskip
 For  $\chi = \chi_0$,  $\Cl_\chi$ is trivial and in this case
 we obtain the supplementary Mirimanoff congruence:
\vspace{-0.2cm}
$$M_{\chi^*}\Big(\Frac{-y}{x} \Big)
= M_{\omega}\Big(\Frac{-y}{x} \Big) = 
\sm_{k= 1}^{p-1}\omega^{-1}(k) \Big(\Frac{-y}{x} \Big)^k \equiv 0 \bmod 
(p)$$
since it corresponds to the $p$-primarity of
$\No_{K/\Q} (x+y\,\zeta)
= z_1^p$.\end{rema}

\subsection{Derivation technics: the method of Eichler}
We begin with a particular case of this method to analyze a global 
approach to the computation of the $p$-rank of the radicals
$W_a$, $W_b$, $W_c$, and $W$.

We consider the necessary condition of the previous subsection, concerning 
the first case of FLT, to have
$\prod_{i=1}^{p-1} (1 + e\,(\zeta^i - 1))^{\lambda_i} \in 
K^{\times p}$, for $e:=\frac{b}{a+b}$:
$$ \sm_{i=1}^{p-1}\Frac{\lambda_i \, i\, X^{i-1}}{1+e(X^i-1)} \in (p, X^p-1). $$
The trick is to suppose that the support $S$ of the set of integers
$\lambda_i$ (that is the set of indices $i$ such that
$\lambda_i \not\equiv 0 \bmod (p)$)
is not too big in the expression:
$$\sm_{i\in S}\lambda_i \, i\, X^{i-1} \prd_{j\in S, j \ne i} (1-e + e\,X^j) 
\in (p, X^p-1), $$

\noindent
so that there is no reduction by $X^p-1$ in the computation of 
the products: 
$$X^{i-1}\prd_{j\in S, j \ne i}(1-e + e\,X^j) ,\ \,\hbox{for}\ i\in S .$$
For this, the condition is that $i-1 + \sm_{j\in S, j \ne i}j < p$,
equivalent to $\sm_{i\in S}i \leq p$.
If we suppose that  $S \subseteq \{1, 2, \ldots, \rho := [\sqrt{2p}\, - 0.5] \}$, the 
condition is satisfied.

\smallskip
We thus have the congruence:
$$\sm_{i\in S}\lambda_i \, i\, X^{i-1} \prd_{j\in S, j \ne i} (1-e + e\,X^j) 
\equiv 0 \bmod p\,\Z_{(p)}[X]. $$
The (unique) term of minimal degree is obtained for the minimal
value $i_0$ of $i\in S$ and gives $\lambda_{i_0}\, .\, (1-e)^{\rho-1}\equiv 0 
\bmod (p)$, then  $\lambda_{i_0}\equiv 0 \bmod (p)$ (contradiction).
We have obtained:

\begin{theo} Let $(a,b,c)$ be a solution in the first case of 
Fermat${}'$s equation.

\smallskip
Then each of the three radicals
$W_a =\langle\,b+c\,\zeta^j\,\rangle.K^{\times p}/K^{\times p}$,
$\ W_b =\langle\,c+a\,\zeta^k\,\rangle.K^{\times p}/K^{\times p}$,
$\ W_c =\langle\,a+b\,\zeta^i\,\rangle.K^{\times p}/K^{\times p}$,
$\ j, k, i =1,\ldots,p-1$, is of $p$-rank at 
least $\rho := [\sqrt{2p}\, - 0.5]$.

Same conclusion replacing $K$ by $K_{\mathfrak p}$ (local radicals).
\end{theo}

But as is always the case, the conclusion 
of the proof is of a local nature.

\begin{rema} (i) The monogenic $\F_p[g]$-module $W_c$ generated by 
$a+b\,\zeta$ defines a subrepresentation of the regular one; thus there 
exist at least $\rho$ distinct characters $\chi$ such that 
$(a+b\,\zeta)^{e_\chi}$ is not a global (or local) $p$th power.
   
\smallskip
 (ii)   Let  $(x_i,y_i) \in
\{(a, b), (b, c), (c, a)\}$,  $i = 1,\ldots,\rho$;
then by the same method it is easy to 
prove that the pseudo-units $x_i + y_i\,\zeta^i$ are independent in 
$K^{\times}/K^{\times p}$, giving by conjugation many subradicals in $W$ of 
$p$-rank $\rho$.\end{rema}

Now we give a variant of the theorem of Eichler from a solution $(a,b,c)$ 
in the first case of the Fermat equation.
 We study the relation, where $e := \frac{b}{a+b}$ 
 (still for the support $S$ of the $\lambda_i$):
$$\prd_{i\in S}\Big(\Frac{a+ b \zeta^{-i}}{a+ b \zeta^{i}}\Big)^{\lambda_i}  =
\prd_{i\in S}\Big(\Frac{ (1+e \,(\zeta^{-i}-1))}{(1+e \,(\zeta^{i}-1))}\Big)^{\lambda_i}  = 
\beta^p, \ \, \beta \in \Z_{(p)}[\zeta]. $$

Put $(a+ b \,\zeta^{i})\,\Z[\zeta] = {\mathfrak c}_i^{\,p}$
and $(a+ b \,\zeta^{-i})\,\Z[\zeta] = \ov{\mathfrak c}_i^{\,p}$. 
From the above relation we deduce:
$$\prd_{i\in S}\Big( \Frac{\ov {\mathfrak c}_i}{{\mathfrak c}_i}\Big)^{\lambda_i}
= (\beta)\,\Z_{(p)}[\zeta]. $$

Reciprocally, any relation of principality
$\prd_{i\in S}\Big( \Frac{\ov {\mathfrak c}_i}{{\mathfrak c}_i}\Big)^{\lambda_i}
= (\beta')\,\Z_{(p)}[\zeta]$
gives:
$$\prd_{i\in S}\Big(\Frac{(1+e \,(\zeta^{-i}-1))}{(1+e \,(\zeta^{i}-1))}\Big)^{\lambda_i}  = 
\zeta^h \, \varepsilon^+\,\beta'{}^p, \ \, \varepsilon^+\in E^+,\, \, 
h \geq 0 ;$$
we suppose that $\sum_{i\in S}\, \lambda_i\, i \equiv 0 \bmod (p)$; this 
implies easily $h=0$. Then the relative norm 
$\No_{K/K^+}(\varepsilon^+\,\beta'{}^p)$
 must be 1, so that $(\varepsilon^+)^2\,\No_{K/K^+}(\beta'{}^p)=1$
 giving $\varepsilon^+ = (\eta^+)^p$ for a real unit $\eta^+$;
 this yields the first relation with $\beta = \eta^+\,\beta'$.
 
\medskip
Write $(1+e \,(\zeta^{-i}-1))^{\lambda_i} = \zeta^{-\lambda_i\,i}
(e + (1-e)\,\zeta^i))^{\lambda_i}$.
The first relation is thus equivalent to the relation (reutilizing by 
abuse the same notations for $F$, $H$, $B$ in $\Z_{(p)}[[X]]$):
$$F(X) := \prd_{i\in S}(e+ (1-e) X^{i})^{\lambda_i} 
\,(1+e(X^{i}-1))^{-\lambda_i} = H(X)^p + B(X)(X^p - 1), $$
giving by logarithmic derivation, $F$ being invertible modulo $(p, 
X^p-1)$:
$$ (1-e) \sm_{i\in S} \Frac {\lambda_i \,i\, 
X^{i-1}}{e+(1-e)X^i}  - e \sm_{i\in S}
\Frac{\lambda_i \,i\, X^{i-1}}{1+e(X^i-1)}  \in (p, X^p-1), $$
and finally:
$$ (1-2 e) \sm_{i\in S}\Frac {\lambda_i \,i\, 
X^{i-1}}{(e+(1-e)X^i)(1+e(X^i-1))}  \in (p, X^p-1). $$
If $2e \equiv 1 \bmod (p)$ we get $a\equiv b \bmod (p)$ and by 
circular permutations, the analogous congruences would give $a \equiv b \equiv 
c \bmod (p)$, thus $0 \equiv a+b+c \equiv 3 a \bmod (p)$ (absurd for 
$p>3$); so 
we may suppose that $2e \not\equiv 1 \bmod (p)$.

\smallskip
As before we obtain:
$$\sm_{i\in S} \lambda_i \,i\, 
X^{i-1}\, \prd_{j \in S, j\ne i} (e+(1-e)X^j)(1 - e + e \, X^j)  \in (p, X^p-1). $$

If $\ 2\sm_{j\in S} j \leq p+1$ there is no 
reduction modulo $X^p-1$ in the computation of this expression.

\smallskip
Then, for  $S \subseteq \{1,\ldots, [\sqrt {p+1}\, - 0.5] \}$,
the (unique) term of minimal degree is obtained for the minimum
$i_0$ of $S$, giving 
immediately $\lambda_{i_0} \equiv 0 \bmod (p)$ (contradiction).

\medskip
Since the classes $\cl({\ov {\mathfrak c}_i} \,.\,{{\mathfrak 
c}_i}^{-1})$ are relative classes, we have proved (taking 
in account that we have imposed a relation on the $\lambda_i$):

\begin{theo}[Eichler${}'$s theorem]  Let $p>2$ be prime.
If the $p$-rank of the relative class group of 
$K$ satisfies ${\rm rk}_p(\Cl^-) \leq \rho' := [\sqrt {p+1}\, - 
1.5]$
then the first case of FLT holds for
the prime $p$.\,\footnote{Since $\cl\big(\frac{\ov {\mathfrak c}_i}{{\mathfrak 
c}_i}\big) = s_i . \cl\big(\frac{\ov {\mathfrak c}_1}{{\mathfrak 
c}_1}\big)$, $i \in S$, the monogenic $g$-module generated by
$\cl\big(\frac{\ov {\mathfrak c}_1}{{\mathfrak 
c}_1}\big)$ contains the $\cl\big(\frac{\ov {\mathfrak c}_i}{{\mathfrak 
c}_i}\big)$ and is contained in the regular representation $\F_p[g]$;
this means that at least $\rho'$ different characters $\chi$
give a nontrivial $\Cl_\chi$ and
the statement is true with the index of irregularity $i(p)$
instead of ${\rm rk}_p(\Cl^-)$.}\end{theo}

\subsection{Some other $p$-adic technics}
Now we consider the Dwork uniformizing parameter
$\varpi$ in $K_{\mathfrak p}$ which has the following characteristic properties (see 
e.g. [Gr1, Exer. II.1.8.3]):

\smallskip
(i) $\varpi^{p-1} = -p$,

\smallskip
(ii) $s_k(\varpi) = \omega(k) \varpi$, $k=1,\ldots,p-1$.

\medskip
In the following lemma we suppose that $1+e \,(\zeta-1)$ is a 
pseudo-unit, so that $p$-primarity and local $p$th power
property are equivalent (see Lem.\,2.1, Th.\,2.2).
We compute in $\Z_p [\zeta] =\Z_p [\varpi]$.
     
\begin{lemm}  Let $\chi = \omega^m$, $m\in \{ 1,\ldots,p-1 \}$; then 
for $e \not \equiv 0 \bmod (p)$ we have the relation $(1+e \,(\zeta-1))^{e_\chi} =1+ \varpi^m 
\varphi_\chi$, where $\varphi_\chi \in \Z_p [\varpi]$.

\smallskip\noindent
Then  $(1+e \,(\zeta-1))^{e_\chi}$
is a local $p$th power if and only if  $\varphi_\chi \equiv
0 \bmod (\varpi)$. \end{lemm}

\begin{proof} Suppose that $(1+e \,(\zeta-1))^{e_\chi} =1+ \varpi^n v$, 
where $v$ is a unit of  $K_{\mathfrak p}$ and $n \geq 1$; put
$v \equiv v_0 \bmod (\varpi)$, $v_0 \in \Z \Sauf p\,\Z$.

Applying $e_\chi$ we have:
 \begin{eqnarray*}
 && (1+e \,(\zeta-1))^{e_\chi} \equiv 
 (1+ \varpi^n v_0)^{e_\chi} \equiv
 1 + e_\chi (\varpi^n v_0) \\
 &\equiv&  1 +  \hbox{$\frac{1}{p-1}$} 
\sm_{j=1}^{p-1} \omega^{-m} (j)\, s_j (\varpi^n v_0) 
 \equiv  1 + \hbox{$\frac{1}{p-1}$} 
\sm_{j=1}^{p-1} \omega^{-m} (j) \,\omega^{n} (j) \,\varpi^n v_0 \\
&\equiv&
   1 + \hbox{$\frac{\varpi^n v_0}{p-1}$}
\sm_{j=1}^{p-1} \omega^{n-m} (j) \equiv 1+ \varpi^n v \ \, \bmod 
(\varpi^{n+1}),
 \end{eqnarray*}
which is absurd except if $n \equiv m \bmod (p-1)$. Thus $(1+e \,(\zeta-1))^{e_\chi}
= 1 + \varpi^m \,\varphi_\chi$.

\smallskip
If $m=p-1$, we know that the norm of such a pseudo-unit is of the form $n^p$
with $n\equiv 1 \bmod (p)$, hence
$(1+e \,(\zeta-1))^{e_{\chi_0}}\equiv
1 \bmod (p^2)$, proving the lemma in this case; suppose $m<p-1$.

\smallskip
The $p$th power condition is $\varphi_\chi \equiv 0 \bmod 
(\varpi^{p-m-1})$ (apply Th.\,2.2), with $p-m-1>0$.

Suppose that  $\varphi_\chi \equiv 0 \bmod (\varpi)$; then  we get
$(1+e \,(\zeta-1))^{e_\chi} = 1+\varpi^{m+1} \varphi'_\chi$, for 
$\varphi'_\chi \in \Z_p[\varpi]$.
 Then applying again the idempotent 
$e_\chi$, the first part of the proof gives
$\varphi'_\chi \equiv 0 \bmod (\varpi)$, then inductively the result
up to  $(1+e \,(\zeta-1))^{e_\chi} \in 1+(\varpi^{m+p-1})$.\end{proof}
 
 The value $m=1$ does not work here since we know
 that $(1+e \,(\zeta-1))^{e_\omega} \equiv 1 + e \,\varpi$ $\bmod 
\  (\varpi^2)$ (see Rem.\,3.4) and since we have supposed $p\notdiv e$.
     
\begin{coro} Write $\ {\rm log}\,(1+e \,(\zeta-1)) =
e_2 \,\varpi^2+ \ldots + e_{p-1}\, \varpi^{p-1}$,
$e_i \in \Z_p$. Then the set of characters $\chi = \omega^m$,
$m \in \{2,\ldots,p-1 \}$, such that $(1+e \,(\zeta-1))^{e_\chi}$
is $p$-primary, is $\{ m \in \{2,\ldots,p-1 \},\ e_m \equiv 0 \bmod (p) \}$.\end{coro}

\begin{proof} Left to the reader.\end{proof}

We see that the condition depends on a single congruence to 0
 modulo $\varpi$, whose probability may be $\frac{1}{p}$,
 giving another aspect of the rarity of such a condition for many 
 values of $m$ (at least $\frac{p-1}{2}$ from Mirimanoff${\,}'$s
 congruences).
 
\subsection{$p$-adic Gauss sums and Mirimanoff${\,}'\!$s polynomials}
 We use the context of the previous Subsection 3.3, especially the 
 Dwork uniformizing parameter $\varpi \in \Q_p(\zeta)$ such that $\varpi^{p-1} = -p$
 and $s_k (\varpi) = \omega(k)\, \varpi$ for $k=1,\ldots,p-1$.
 
 \smallskip
Let $\chi = \omega^m$, here indexed by $m \in\{ 0,\ldots,p-2 \}$.
 We note that (additively):
$$e_\chi \,.\, \zeta = \Frac{1}{p-1} \sm_{k=1}^{p-1} \chi^{-1}(k) 
\zeta^k = \Frac{-1}{p-1} \tau(\chi^{-1}). $$
Now put:
$$\zeta = \Frac{-1}{p-1}\big( u_0 + u_1\, \varpi+ \ldots + 
u_{p-2}\,\varpi^{p-2}\big) ,$$
where $u_k \in \Z_p$, with $u_0 \equiv 1 \bmod (p)$.
We know that $e_\chi \,.\,\varpi^{j} = 0$ if $j \not \equiv m$ 
modulo $(p-1)$ and $e_\chi \,.\,\varpi^{m} = \varpi^{m}$, so that
$e_\chi \,.\, \zeta =  \hbox{$\frac{-1}{p-1}$} u_m \,\varpi^{m}$
and $\tau(\chi^{-1}) =  u_m \,\varpi^{m}$,
for all $m\in \{0, \ldots, p-2\}$.\,\footnote{In these computations, we must 
write the unit character $\omega^0$ instead of $\omega^{p-1}$
because of the expression of $\zeta$ since $\varpi^{p-1} = -p$
and $\tau(\omega^{0}) = 1$.}
 Then, since for $\ov\tau(\chi) := \chi(-1)\,\tau(\chi)$, we have
$\tau(\chi^{-1})\,\ov\tau(\chi) = p$ for $\chi \ne \chi_0$, we obtain for $m \ne 0$:
$$u_m \,\varpi^{m} \, (-1)^{m} \, u_{p-1-m}\,\varpi^{p-1-m} = 
 (-1)^{m} \, u_m \, u_{p-1-m} \,(-p) = p, $$
 giving the  relation $u_m\, u_{p-1-m} = (-1)^{m+1}$, for $m \ne 0$.
 For the unit character, $\tau(\chi_0) = 1$ and we find $u_0 = 1$.

 \smallskip
We have obtained a classical result:

\begin{prop}  Let $\chi = \omega^m$, $m \in 
\{0,\ldots,p-2 \}$, 
 $\tau(\chi^{-1}) := -\sum_{k=1}^{p-1} \chi^{-1}(k)\, \zeta^k$
the Gauss sum of $\chi^{-1}$;  put $\ov\tau(\chi) := -\sum_{k=1}^{p-1} 
 \chi(k)\, \zeta^{-k} = \chi(-1)\,\tau(\chi)$.

\smallskip
 Then we have $\tau(\chi^{-1}) =  u_m\, \varpi^{m}$,
  $\ov \tau(\chi) = \chi(-1)\, u_{p-1-m} \,\varpi^{p-1-m}$, which 
  implies the   relation
$u_m\, u_{p-1-m} = (-1)^{m+1}$, for all $m\in \{1, \ldots, p-2\}$, and $u_0 = 1$.\end{prop}

The modified Mirimanoff polynomial is, for $\chi^* = \omega\chi^{-1}$
(see Rem.\,3.6, (ii)):
$$ M_{\chi^*}(Z)  := \sm_{k=1}^{p-1}(\chi^*)^{-1}(k)\, Z^k, $$
and in the first case of FLT we must compute
$M_{\chi^*}(\frac{-x}{y})$ modulo $(p)$ for the usual $(x,y)$ depending of a 
solution and its permutations. 

\medskip
We suppose now that $\omega$ takes its values in
the field $F$ of $(p-1)$th roots of unity.
We consider the ideal ${\mathfrak p}_0\,\vert\, p$ of $F$ such that
$\omega(k) \equiv k \bmod {\mathfrak p}_0$ for all $k$.
All the computations take place in the compositum $F K$ in which
we denote by $\mathfrak P$ the (unique) prime ideal above
${\mathfrak p}_0$. 
 
\medskip
The condition of $p$-primarity of
$(a+b\,\zeta)^{e_\chi}$, for $\chi = \omega^m$, $m\in \{1, \ldots, p-1\}$,
$\chi \ne \omega$
(see Subsec.~3.1) becomes, in $FK$ with $\chi^*= \omega^{1-m}$:
$$M_{\chi^*} \Big(\Frac{-b}{a}\Big) \,.\, \tau(\chi^*)
\equiv 0 \bmod {\mathfrak P}^{p-1}, $$
where $M_{\chi^*} \big(\frac{-b}{a}\big) \in F$ and
$\tau(\chi^*) := -\sum_{k=1}^{p-1} \chi^*(k) \zeta^k \in FK$ is of
${\mathfrak P}$-valuation $m-1$, giving 
$M_{\chi^*} \big(\frac{-b}{a}\big) \equiv 0 
\bmod {\mathfrak p}_0$, and where we have:
$$M_{\chi^*}(\zeta) = - \tau((\chi^*)^{-1}). $$
The above properties of Gauss sums lead to the following, where
we only suppose that $a$ and $b$ are coprime integers.

\medskip
Put $(a + b \,\zeta)\,\Z[\zeta] = {\mathfrak C}_1$, thus $\zeta \equiv 
\frac{-a}{b} \bmod {\mathfrak C}_1$ seen in $FK$.
This gives:
$$-M_{\chi^*} \Big(\Frac{-a}{b}\Big) = -\sm_{k=1}^{p-1}(\chi^*)^{-1} (k) 
\,\Big(\Frac{-a}{b}\Big)^k \equiv \tau((\chi^*)^{-1})\bmod {\mathfrak C}_1 .$$ 
Thus in the same way (using $\zeta^{-1} \equiv 
\frac{-b}{a} \bmod {\mathfrak C}_1$):
$$-M_{(\chi^*)^{-1}} \Big(\Frac{-b}{a}\Big) \equiv \ov \tau(\chi^*)
\bmod {\mathfrak C}_1, $$
which yields, in $F$,
for any $\chi \ne \omega$ (i.e., $\chi^* \ne \chi_0$):
$$M_{\chi^*} \Big(\Frac{-a}{b}\Big) \,.\,M_{(\chi^*)^{-1}} \Big(\Frac{-b}{a}\Big) \equiv
\tau((\chi^*)^{-1})\,.\, \ov\tau(\chi^*) \equiv p \bmod {\mathfrak C}_1. $$
 Let $\sigma$ be an element of ${\rm Gal\,}(FK/K)$; for any $(p-1)$th 
 root of unity $\xi$,  $\sigma (\xi) = \xi^t$ with a suitable $t$ prime to $p-1$,
 so that the action of $\sigma$ on the powers of $\omega$ preserves 
 the relation $\varphi\,.\,\varphi^{-1} = \chi_0$
 between the characters, and preserves the ideal ${\mathfrak C}_1$
 which is in $K$; thus the expressions 
 $M_{\chi^*} \big(\frac{-a}{b}\big) \,.\,M_{(\chi^*)^{-1}}
 \big(\frac{-b}{a}\big)$ are conjugated by Galois
 so that the $p$-adic study\,\footnote{More precisely the knowledge of 
 the  ${\mathfrak p}'_0$-valuations, for all the prime ideals ${\mathfrak p}'_0$ of 
 $F$ above $p$.}
 of the products $M_{\omega^d} \big(\frac{-a}{b}\big) 
 \,.\,M_{\omega^{-d}} \big(\frac{-b}{a}\big)$, $d\,\vert\, p-1$,
 is sufficient.
 
\smallskip
The congruence modulo ${\mathfrak C}_1$ in $FK$
is now in $F$, thus  it is actually modulo the 
ideal $\No_{K/\Q} ({\mathfrak C}_1)$ seen in $F$. 
 Since it is the norm 
 of $a + b \,\zeta$, it is the homogeneous form in $a$, $b$:
 $$\Phi_p (a,b) := a^{p-1} - a^{p-2} b + \ldots - a\,b^{p-2} + b^{p-1}. $$
 Put $M_{\chi^*} \big(\frac{-a}{b}\big) \,.\,M_{(\chi^*)^{-1}} 
 \big(\frac{-b}{a}\big) - p 
 = \Phi_p (a,b) \,.\, \Frac{\Psi_\chi(a,b)}{a^{p-2}\, b^{p-2}}$, then
 $\Psi_\chi(a,b)$ is an homogeneous form of degree $p-3$.
 
 \smallskip\smallskip
 We have, for any character $\varphi$, $M_{\varphi}(Z) = \varphi(-1)\,Z^p\,M_{\varphi}(Z^{-1})$,
 which gives
$M_{\varphi}(Z) \,M_{\varphi^{-1}}(Z^{-1}) = M_{\varphi}(Z^{-1})\,
 M_{\varphi^{-1}}(Z) $, hence proves the symmetry 
 between $a$ and $b$, and the invariance of
 $M_{\chi^*} \big(\frac{-a}{b}\big) \,.\,M_{(\chi^*)^{-1}}
 \big(\frac{-b}{a}\big)$ by complex conjugation
  in $F/\Q$. So these  expressions have coefficients
 in the maximal real subfield $F^+$ of $F$.

 \smallskip
To summarize, we have obtained:
 
\begin{prop} Let $x$, $y$ be  indeterminates 
and put $M_{\varphi}(Z) :=  \sum_{k=1}^{p-1} \varphi^{-1}(k) 
\,Z^k$ for any character $\varphi$.
Then for all $\chi\ne \omega$, we have the relation:
$$M_{\chi^*} \Big(\Frac{-x}{y}\Big) \,.\, M_{(\chi^*)^{-1}} \Big(\Frac{-y}{x}\Big) =  p 
+ \Phi_p (x,y) \,.\, \Frac{\Psi_\chi(x,y)}{x^{p-2}\, y^{p-2}} \,, $$
where $\Psi_\chi(x,y)$ is a symmetrical homogeneous form of degree $p-3$ with 
coefficients in $F^+$.\end{prop}
  
Now we suppose that $(x, y, z)$ is a solution in the first case
of the Fermat equation.
Recall that the condition of $p$-primarity of $(x+y\,\zeta)^{e_\chi}$ which was 
modulo ${\mathfrak P}^{p-1}$ in $FK$  is now, because of the total 
ramification in $FK/F$,  modulo the prime ideal ${\mathfrak p}_0$ of $F$ under
${\mathfrak P}$, and is $M_{\chi^*} \big(\frac{-y}{x}\big)
\equiv 0 \bmod {\mathfrak p}_0$. 

\smallskip
From the above we obtain that:
$$M_{\chi^*} \Big(\Frac{-x}{y}\Big) \,.\,
M_{(\chi^*)^{-1}} \Big(\Frac{-y}{x}\Big) \equiv 0 \bmod {\mathfrak p}_0$$
is  equivalent to $\Psi_\chi(x,y) \equiv 0 \bmod {\mathfrak p}_0$.

\smallskip
For instance, for $p=5$ we have (noting that $F^+ = \Q$
and ${\mathfrak p}_0 = (5)$):
 \begin{eqnarray*}
M_{\omega^{-1}} \Big(\Frac{-x}{y}\Big) \,.\, M_{{\omega}} 
\Big(\Frac{-y}{x}\Big) &=&  5
+ \Phi_5 (x, y) \,.\, \Frac{x^2 + x\, y + y^2}{x^{3}\, y^{3}}, \\
M_{\omega^{2}} \Big(\Frac{-x}{y}\Big) \,.\, M_{\omega^{2}}
\Big(\Frac{-y}{x}\Big) &=&  5
- \Phi_5 (x, y) \,.\, \Frac{x^2 +3\, x\, y + y^2}{x^{3}\, y^{3}}.
\end{eqnarray*}
Of course these forms $\Psi$ do not represent 0 in $\F_5$.
 
\begin{rema} (i) Notice that these congruences have 
nothing to do with Mirimanoff${}'$s congruences despite the fact that
as soon as one of the factors $M_{\chi^*} \big(\frac{-x}{y}\big)$,
$M_{(\chi^*)^{-1}} \big(\frac{-y}{x}\big)$ is congruent to 0
modulo ${\mathfrak p}_0$, this is the case of the expression $\Psi_\chi(x,y)$
and reciprocally.

\smallskip
More precisely, $M_{\chi^*} \big(\frac{-x}{y}\big) \equiv 0 
\bmod {\mathfrak p}_0$ is equivalent to $(y+x\,\zeta)^{e_\chi}$
$p$-primary,
hence to $(x+y\,\zeta)^{e_\chi}$ $p$-primary (since $\chi \ne \omega$), 
thus to $M_{\chi^*} \big(\frac{-y}{x}\big) \equiv 0 
\bmod {\mathfrak p}_0$.

\smallskip
Similarly,
 $M_{(\chi^*)^{-1}} \big(\frac{-y}{x}\big)  \equiv 0 
\bmod {\mathfrak p}_0$ is equivalent to the $p$-primarity of
the two pseudo-units
$(x+y\,\zeta)^{e_{\wt \chi}}$ and $(y+x\,\zeta)^{e_{\wt \chi}}$, 
then to $M_{(\chi^*)^{-1}} \big(\frac{-x}{y}\big)  \equiv 0 
\bmod {\mathfrak p}_0$,
where $\wt \chi:= \omega^2\,\chi^{-1}$,
which may have some interest (see in Subsec. 2.5, (b), the reflection 
between $R_{2,\chi}$ and ${\mathcal T}_{{\omega^2 \chi}^{-1}}$). 

\smallskip
(ii) It would be interesting to perform the same study 
with the Davenport--Hasse relations between Gauss sums,
for two characters:
$$\prd_{\chi,\,\chi^d=\chi_0}\tau(\chi\,.\,\psi) = \psi^{-d}(d)
\,.\,\tau(\psi^d)\,.\prd_{\chi,\,\chi^d=\chi_0}\tau(\chi),$$
for any divisor $d$ of $p-1$,
and with the Jacobi sums given by the relation:
$$\Frac{\tau (\chi)\,\tau (\psi)}{\tau (\chi\,\psi)}
= -\sm_{k=1}^{p-1} \chi(k) \,\psi(1-k). $$ \end{rema}

\subsection{Mirimanoff${\,}'\!$s sums}
We still consider the context of the previous Subsection 3.4,
for which $\omega$ takes its values in
the field $F$ of $(p-1)$th roots of unity.
We fix the prime ideal ${\mathfrak p}_0$ of $F$ above $p$ in the 
following way: fix a primitive $(p-1)$th root of unity $\xi_0 \in F$ and a
primitive $(p-1)$th root   $r_0 \in \Z$ modulo $p$; then we decrete that
$\xi_0 \equiv r_0 \bmod {\mathfrak p}_0$.

\smallskip
Since for any character $\varphi$ of $g := {\rm Gal\,}(K/\Q)$,
$M_\varphi(Z) = \sum_{k=1}^{p-1}\varphi^{-1}(k)\, Z^k$, if we put,
 for a solution $(x, y, z)$ in the first case of  Fermat${}'$s 
equation:
$$\Frac{-y}{x} \equiv r_0^t  \equiv \xi_0^t =:\xi \bmod {\mathfrak p}_0, $$
we have in $F$ the congruence:
$$M_\varphi\Big(\Frac{-y}{x}\Big) \equiv M_\varphi(\xi) \bmod {\mathfrak 
p}_0; $$
hence the congruences $M_\varphi\big(\frac{-y}{x}\big) \equiv 0 \bmod 
(p)$ in $\Q_p$ and $M_\varphi(\xi) \equiv 0 \bmod {\mathfrak p}_0$ 
in $F$ are equivalent.

\smallskip
We propose to call the sums of roots of unity:
$$\mu_\varphi(\xi) := 
\sm_{k=1}^{p-1}\varphi^{-1}(k)\, \xi^k \,\in \, F ,$$
the Mirimanoff sums attached to the character $\varphi$ and the 
$(p-1)$th root of unity~$\xi$.

\smallskip
It is clear that the algebraic numbers:
$$\mu_\varphi(\xi)\,.\,\mu_{\varphi^*}(\xi), \ \varphi \ne \chi_0,\, 
\omega \ \,
{\rm and}\ \  \mu_\varphi(\xi)\,.\,\mu_{\varphi^{-1}}(\xi^{-1}), \ \varphi \ne \chi_0, $$
give the easy way to study the congruences of Mirimanoff and the
congruences given in Proposition 3.18.

Unfortunately, the root $\xi$ is uneffective and the properties of 
the sums $\mu_\varphi(\xi)$ depend largely of the order of $\xi$
(i.e., the order of $\frac{-x}{y}$ modulo $p$); hence we must envisage 
all the possibilities.

\medskip
Warning: in the factor $\varphi^{-1}(k)$, $k$ is considered modulo 
$p$, but in the factor $\xi^k$, $k$ is considered modulo 
$p-1$, under the condition that $k \in \{1,\ldots,p-1\}$.

\smallskip
In a more numerical setting, put $\varphi = \omega^h$ and $\xi = 
\xi_0^t$; then, writing $k \equiv  r_0^j \bmod (p)$,
we get: 
 \begin{eqnarray*}
\mu_\varphi(\xi) =: \mu_h(t) &=&
\sm_{k=1}^{p-1}\omega^{-h}(k)\,\xi_0^{\,t\,k} =
\sm_{j=1}^{p-1}\xi_0^{-h\,j}\,\xi_0^{\,t\,[ r_0^j]_p} \\
&=& \sm_{j=1}^{p-1}\xi_0^{-h\,j\,\, +\,\, t\,[ r_0^j]_p} ,\,\  
h, \,t\in \{1,\ldots,p-1\},
\end{eqnarray*}
where $[r_0^j]_p$ is the unique residue modulo $p$ of $r_0^j$ in the 
set $\{1,\ldots,p-1\}$.

\smallskip
Then let $\Phi_{p-1}$ be the $(p-1)$th cyclotomic polynomial, of degree
$\nu := \phi(p-1)$; after reduction modulo $\Phi_{p-1}$, we obtain:
$\mu_h(t) = q_0 + q_1 \xi_0 + \ldots + q_{\nu-1} \xi_0^{\nu-1}$, $q_i 
\in \Z$,
which can be studied modulo ${\mathfrak p}_0$ in an easy way.

\smallskip\smallskip
Naturally, these sums are completely analogous to Mirimanoff${}'$s 
polynomials specialized at suitable classes modulo $p$, but we hope that
the formulation in terms of sums of roots of unity is likely of a better understanding.

\subsection{Wieferich${}'$s criterion: a local consequence of the reciprocity law}
As indicated in Ribenboim${}'$s book,
the Wieferich criterion may be deduced from the law of reciprocity (this has been done
first by Furtw\"angler from Eisenstein${}'$s reciprocity law [R, IX.3]).
For this purpose, an explicit formula of Hasse may also be used [R, IX.5].

\medskip
Here we propose a more basic  proof
using the ${\mathfrak p}$-conductor of a Kummer extension
in the following way, where 
$\big(\frac{\bullet}{\bullet}\big)_{\!p}$ is the $p$th power residue 
symbol, with values in $\langle \,\zeta\, \rangle$.
 
\begin{theo}[Wieferich${}'$s criterion]  Let $\ell$ be a prime 
number, $\ell \ne p$, and suppose that $x + y \,\zeta$ is a pseudo-unit
(i.e., $(x + y \,\zeta)$ is the $p$th 
power of an ideal of $K$ prime to ${\mathfrak p}$).

(i) Then $\Big(\Frac{\zeta^x \, y + \zeta^{-y} \, x}{\ell}\Big)_{\!p} = 1$. 

(ii) If $\ell\,\vert\, y$ with $p \notdiv y$ and if $(x, y, z)$ is a 
solution of Fermat${}'$s equation\,\footnote{So that $x+y = z_0^p$ as usual; the 
second case of FLT being equivalent here to $p\,\vert\,x$.}\!\!,
then $\ell^{\,p-1} \equiv 1 \bmod (p^2)$.   \end{theo}

\begin{proof} The expression of $\alpha := \zeta^x \, y + \zeta^{-y} \, x$ is such 
that $\alpha$ is still a pseudo-unit, and $\alpha \equiv 
x+y \equiv (x+y)^p \bmod (1-\zeta)^2$.

The general law of reciprocity (see e.g. [Gr1, Th.\,II.7.4.4]) yields to:
$$\Big(\Frac{\alpha}{\ell}\Big)_{\!p}\,\Big(\Frac{\ell}{\alpha}\Big)^{-1}_p 
= (\ell, \alpha)_{\mathfrak p}$$
where $(\bullet, \bullet)_{\mathfrak p}$ is the Hilbert${}'$s symbol at the 
place ${\mathfrak p}$. This symbol is equal to 1 if and only if
$\ell$ is a local norm in the Kummer extension 
$K_{\mathfrak p}(\sqrt[p]\alpha\,)/K_{\mathfrak p}$; the conductor
of this extension divides ${\mathfrak p}^{p-1}$ since $\alpha$
is congruent to a $p$th power modulo ${\mathfrak p}^2$
(see the general conductor formula in [Gr1, Prop.\,II.1.6.3]).
Since $\ell^{\,p-1} \equiv 1 \bmod (p)$ the normic condition is satisfied for $\ell$.

\smallskip
But the symbol $\big(\frac{\ell}{\alpha}\big)^{-1}_p$ is tivial 
since $(\alpha)$ is the $p$th power of an ideal; thus:
$$\Big(\Frac{\zeta^x \,y + \zeta^{-y} \, x}{\ell}\Big)_{\!p} = 1. $$
If $\ell\,\vert\, y$,
we have $\zeta^x \,y + \zeta^{-y}\,x \equiv \zeta^{-y}\,x 
\  \bmod (\ell)$ and $1 = \big(\frac{\zeta^{-y}\,x}{\ell}\big)_{\!p} =
\big(\frac{\zeta}{\ell}\big)^{-y}_p\,\big(\frac{x}{\ell}\big)_{\!p}$; but 
$x =  z_0^p - y
\equiv  z_0^p \bmod (\ell)$ giving $\big(\frac{x}{\ell}\big)_p = 1$ and
$\big(\frac{\zeta}{\ell}\big)_p = 1$ since $p\notdiv y$.

If $(\ell) = {\mathfrak l}_1\ldots {\mathfrak l}_d$ in $K$,
then $\prod_{i=1}^d \big(\frac{\zeta}{{\mathfrak l}_i}\big) = 1$, but 
we have $ \big(\frac{\zeta}{{\mathfrak l}_1}\big)^k = 
s_k \,\big(\frac{\zeta}{{\mathfrak l}_1}\big)=
\big(\frac{\zeta^k}{{\mathfrak l}_k}\big) =
 \big(\frac{\zeta}{{\mathfrak l}_k}\big)^k$,
 so that $\big(\frac{\zeta}{{\mathfrak l}_k}\big)$ does not depend on 
 $k$, giving $\big(\frac{\zeta}{{\mathfrak l}_1}\big) =1$; thus the 
 multiplicative group of the residue field of ${\mathfrak l}_1$
 contains an element of order $p^2$, proving the point (ii) of the 
 theorem. \end{proof}
 
Then the discovery of Wieferich${}'$s criteria consists in proving that small 
prime numbers $\ell$ (e.g. $\ell = 2$) divide $a\,b\,c$ (see [GM], 
[Th2], for a study of Fermat${}'$s quotients in relation with FLT);
in the second case, the hypothesis $\ell \,\vert\, y$, $p\notdiv y$ may be inaccurate,
so the  Wieferich criterion is uneffective in the second case.

\smallskip
It is clear that the prime numbers $\ell\equiv 1 \bmod (p)$, such that Fermat${}'$s equation
$u^p + v^p+1 = 0$ has no nontrivial
solutions in the finite field $\F_\ell$, are 
divisors of $a \,b\, c$ (where $(a, b, c)$ is a global solution in any case of 
Fermat${}'$s equation); then experimental computations show that many such primes do 
exist. One may conjecture that their number tends to infinity with 
$p$, which gives many uneffective Wieferich${}'$s criteria.

\smallskip
In this direction we have the following interesting approach.

\subsection{Wendt${\,}'\!$s criterion: a non modulo $p$ local--global result}
Let $\ell$ be a prime number of the form $1 + n \,p$, $n \geq 2$,
and let $\mathfrak l$ be an ideal above $\ell$ in~$K$. 
We consider the algebraic number
$\theta_n := \prod_{i,\,j = 1}^n (\xi_i + \xi_j + 1)$,
where the $\xi_k$, $k=1,\ldots,n$, are the  $n$th roots of unity.

\smallskip
We have $\theta_n \in \Z \Sauf \{ 0 \}$; this number has 
been used for instance in the following papers : [LS] (for a similar purpose 
as us) and [A-HB], [F] to prove that the first case of FLT holds for 
infinitely many primes $p$.

\smallskip
See [R, IV.4] for its explicit computation via Wendt${}'$s determinant.
If $\ell \notdiv \theta_n$ this means that Fermat${}'$s equation in the 
residue field $\F_{\ell}$ of $\mathfrak l$ has no nontrivial 
solutions; thus if $a$, $b$, $c$ is a solution in $\Z$ of
Fermat${}'$s equation,
necessarily $\ell$ divides one of these numbers, say $\ell \,\vert\,c$.

\smallskip
Now we state the following result (in the spirit of  Germain${}'$s 
theorem).

\begin{theo}[Wendt${}'$s criterion]  Let $\ell = 1 + n \,p$ be
a prime number which 
does not divide the natural integer $\theta_n$. Moreover, we suppose 
that $p$ is not a $p$th power modulo~$\ell$.

\smallskip\noindent
Then the first case of FLT holds for $p$.\end{theo}

\begin{proof} Suppose that $\ell \,\vert \,c$ for a solution
in the first case of  Fermat${}'$s equation.
We have $a+b = c_0^p$, $\No_{K/\Q} (a+b\,\zeta) = c_1^p$
with $-c = c_0 \, c_1$ (see Rem.\,1.4, (i)).

If $\ell \,\vert \, c_0$
then $b \equiv -a \bmod (\ell)$, giving:
$$c_1^p = \No_{K/\Q} (a+b\,\zeta) \equiv a^{p-1} \prd_{i=1}^{p-1}(1-\zeta^i)
= a^{p-1}\, p \ \bmod (\ell)$$
(a contradiction since $a+c \equiv a \equiv b_0^p \bmod (\ell)$, 
giving that $p$ is a local $p$th power at $\ell$).

\smallskip
So $\ell \,\vert \, c_1$; from Lemma 1.2,  $\ell \,\notdiv \, c_0$
giving, from $a + b = c_{0}^p$,
$a+c = b_0^p$, and $b+c = a_0^p$, the relation
$0 = a + b - c_{0}^p \equiv  b_0^p +  a_0^p + (- c_{0})^p \bmod (\ell)$ which 
defines a non trivial solution in $\F_\ell$ (absurd).

The conclusion of the theorem is the same if we replace the hypothesis
``$\,p$ is not a $p$th power modulo $\ell\,$'',
by ``\,$p\notdiv n$\,'' since in that case,
Wieferich${}'$s criterion is not satisfied for $\ell$.\end{proof}


{\bf Appendix. Wieferich${}'$s criterion without
reciprocity law} (from a proof rediscovered by {\sc \small  Roland Qu\^eme).}\,\footnote{
Adress: Roland Qu\^eme, 13 Avenue du ch\^ateau d'eau, 31490 Brax,
Url: http://roland.queme.free.fr/, email: roland.queme@wanadoo.fr}

\medskip\smallskip
We use the same notations as in Subsection 3.6.
See also Notations 2.7.

Let $\ell \ne p$ be a prime number.
We suppose that 
by choosing suitable $x$, $y$ among $a$, $b$, $c$, we have
$\ell\, \vert\, y$ and $p \notdiv x+y$ in the writing 
$(x+y\,\zeta)\,\Z[\zeta] = {\mathfrak z}^p_1$ (valid in any case
of Fermat${}'$s equation). Consider $e_{\omega}\in \Z[g]$ modulo $p$.

\smallskip
We know that $\cl({\mathfrak z}_1)^{e_{\omega}} = 1$
(another application of the reflection theorem; see 
 [Gr1, II.5.4.6.3]), so that
$(x+y\,\zeta)^{e_{\omega}} = \varepsilon_\omega\, \delta_\omega^p$, 
$\varepsilon_\omega \in E_\omega = \langle\zeta\rangle$,
$\delta_\omega \in K^\times$; hence
$\varepsilon_\omega = \zeta^h$ for $h \geq 0$.

Thus this yields:
 $$(x+y\,\zeta)^{e_{\omega}}\,\in\, \zeta^h\,.\,K^{\times p}, $$

hence  the relation
 $\big(\frac{(x+y\,\zeta)^{e_{\omega}}}{\ell}\big)_p \!\!=
 \big(\frac{\zeta}{\ell}\big)^h_p$ where
$(x+y\,\zeta)^{e_{\omega}} \equiv x^{e_{\omega}}$ (a $p$th power)
modulo $\ell$, proving that:
$$\Big(\Frac{\zeta}{\ell}\Big)^h_p = 1.$$

But $(x+y\,\zeta)^{e_{\omega}}\,\in\, \zeta^h\,.\,K^{\times p}$ is
equivalent to 
$(1+\hbox{$\frac{y}{x+y}$}(\zeta-1))^{e_{\omega}}\,\in\, \zeta^h\,.\,K^{\times p}$;
 using Remark 3.4 (\,$(1+\frac{y}{x+y}(\zeta-1))^{e_{\omega}}\equiv
 1+\frac{y}{x+y}(\zeta-1) \bmod {\mathfrak p}^2$) we get immediately
 $h \equiv \frac{y}{x+y} \bmod (p)$.
 
 \smallskip
 If moreover $y \not \equiv 0 \bmod (p)$ (e.g. first case of FLT, 
 or second case with $x\equiv 0 \mod (p )$)
 we obtain the result on Wieferich${}'$s criterion in the same way as 
 in Subsection 3.6, without any use of the reciprocity law.

 
\section{Conclusion}

\medskip\smallskip
We have shown that much of the classical literature on FLT has been concerned with
very basic facts of class field theory, often rediscovered by means of painful
congruential computations; but recall that class field 
theory is essentially algebraic as soon as, for instance, \v Cebotarev${}'$s density 
theorem is not used (among other analytic tools),
and that, algebrically, all is ``\,possible\,''.
So it appears that this approach
is relatively poor, despite the power of class field theory to enunciate
technical properties.

\smallskip
Moreover, most of the arguments are local, especially {\it local at 
$p$}.\,\footnote{Recall that a pseudo-unit $\alpha$ of $K$ is in 
$K^{\times p}$ if and only if $\alpha \in K_{\mathfrak q}^{\times p}$
for all ${\mathfrak q} \in \{{\mathfrak p}, {\mathfrak l}_1, \ldots
{\mathfrak l}_r\}$, where the prime ideals ${\mathfrak l}_1, \ldots
{\mathfrak l}_r$ generate the $p$-class group of $K$ (see [Gr1, 
Exer. II.6.3.8]); but this criterion is not effective. }

\smallskip
The fact that the relative class group takes place in these studies 
does not change our point of view since it is utilized without serious
analytic arguments (except the unusable upperbound
${\rm log}\,(h^-) < \hbox{$\frac{p}{4}$} {\rm log}(p)$ and the ingenious 
but elementary derivation technic of  Eichler). Moreover
the analytic class number formula for the relative class group is not 
really analytic since it is, roughly speaking, equivalent to
Stickelberger${}'$s theorem and is, in some sense, algebraic
(the main theorem on cyclotomic 
fields gives a better knowledge of the class field theory aspects, but 
it is not really necessary).

\smallskip
It is likely that the most serious {\it cyclotomic\,} approaches
are the study of ``\,Mirima\-noff${}'$s sums\,'', since at least
half of them must be zero modulo ${\mathfrak p}_0$, and 
that of Wendt${}'$s criterion since it is connected with the theory 
of prime numbers; but all this only concerns  the first 
case of FLT, which is unnatural.

\smallskip
Still in the first case, from the well-known class field theory 
exact sequence of $\Z_p$-modules:
$$1 \too U/\ov E \tooo{\rm Gal}(H_{\Pl_p}/K) \tooo \Cl \too 1, $$
where $H_{\Pl_p}$ is the maximal abelian $p$-ramified 
pro-$p$-extension of $K$, $U$ the group of principal units of
$K_{\mathfrak p}$, $\ov E$ the closure in $U$ of the group of global 
units $\varepsilon\equiv 1 \bmod {\mathfrak p}$, we get for any {\it even}
character $\chi \ne \chi_0$:
$$1 \too U_\chi/\ov E_\chi \tooo {\mathcal T}_\chi \tooo \Cl_\chi \too 1, $$
where all groups are $p$-torsion groups since $\chi \ne \chi_0$ is 
even.
For $p$ large enough, the result of Kurihara--Soul\'e is 
$\,\Cl_{\omega^{p-3}} = 1$; suppose that it is possible to extend it to
${\mathcal T}_{\omega^{p-3}} = 1$ (taking ``\,$p$-ramification\,'' instead of 
``\,nonramification\,''), then $\ov E_{\omega^{p-3}}= U_{\omega^{p-3}}$
which means that the fundamental ${\omega^{p-3}}$-unit 
$\varepsilon_{\omega^{p-3}}$ is not a local $p$th power and that the 
fundamental ${\omega^{p-3}}$-cyclotomic unit $\eta_{\omega^{p-3}}$
(equal to $\varepsilon_{\omega^{p-3}}$ since $\Cl_{\omega^{p-3}} = 1$)
is not a local $p$th power, which is equivalent to $b_{\chi^*} =
b_{\omega^{3}} \not \equiv 0 \bmod (p)$, in other words to
$B_{p-3}\not \equiv 0 \bmod (p)$, which would contradict the first 
case of FLT (at least for $p$ large enough).

\medskip
We believe more in the possibility of a {\it nonalgebraic} study of the 
radical generated by $\zeta$, $1-\zeta$, $a+b\,\zeta$, $b+c\,\zeta$, $c+a\,\zeta$ and 
their conjugates, which would be independent of the considered case of FLT, 
and which is not equivalent to a general study of the 
group ${}_p\Cl$ because as a matter of fact
we are concerned with very specific $p$-classes,
the same remark being valid for the 
utilization of other arithmetical invariants of $K$.
As the Referee mentions,
all these invariants are isomorphic or dual to adequate Tate twists
of the cohomology group ${\rm H}^2 ({\mathcal G}, \Z/p\Z)$
(where ${\mathcal G}$ is the Galois group of the maximal $p$-ramified
pro-$p$-extension of $K$) which relativizes the interest,
but we don${}'$t know if the use of the 
pseudo-units $x+y\,\zeta$ in these contexts
leads, in practice, to the same ``\,numerical\,'' criteria and to the 
same diophantine approach. 

\smallskip
It is indeed surprising that, to our knowledge, there is no
important diophantine results on 
the mixed radical $W$, using simultaneously $a, b, c$, and possibly 
the cyclotomic numbers,
which constitutes a particular case of the study of the polynomial 
identity, in the polynomial ring $\Z[X]$:
$$\prd_{i = 1}^n (u_i + v_i\,X^{d_i})^{\lambda_i} =
H(X)^p + B(X)\,(X^p - 1),\ \ 0\leq d_i , \,\lambda_i \leq p-1 . $$

\end{document}